\documentclass[onefignum,onetabnum]{siamart220329}
\usepackage{amssymb}
\usepackage{enumitem}
\usepackage{booktabs}
\usepackage{xcolor}

\newtheorem{remark}[theorem]{Remark}
\newtheorem{example}{Example}[section]
\newcommand{\R}{\mathbb{R}}

\newcommand{\bx}{\boldsymbol{x}}

\headers{GPU-Accelerated Direct Solver for the Schr\"odinger Operator}{X.\ Liu and X.\ Zhang}

\title{A Simple GPU-Accelerated Solver for the Schr\"odinger Operator
  with Applications to Ground States and Hamiltonian Simulation}

\author{Xinyu Liu\thanks{Department of Mathematics,
  The Ohio State University, 231 West 18th Avenue, Columbus,
  OH 43210 (\email{liu.12165@osu.edu}).}
\and Xiangxiong Zhang\thanks{Corresponding author. Department of Mathematics,
  Purdue University, 150 N.\ University Street, West Lafayette,
  IN 47907 (\email{zhan1966@purdue.edu}).}}

\begin{document}
\maketitle

\begin{abstract}
We extend the tensor-product direct solver from the Laplacian to
the Schr\"odinger operator $-\Delta + V$.
When the potential $V_1$ is separable, the operator
$-\Delta + V_1$ is inverted or exponentiated at cost $O(N^{1+1/d})$
in $d$ dimensions via per-axis eigendecomposition.
On a single NVIDIA A100 GPU, this costs less than one second
for $10^9$ degrees of freedom  in 3D.
For non-separable potentials $V = V_1 + V_2$, the same solver
provides a preconditioner $(-\Delta + V_1)^{-1}$ for the
preconditioned conjugate gradient (PCG) method
and a propagator for operator-splitting time integrators.
For bounded $V_2$, we prove that the preconditioned operator
has a bounded condition number and a clustered spectrum with at most finitely many outlier
eigenvalues, independently of the mesh size, and also independently
of the domain size when $V_1$ is a confining potential.
This explains the mesh- and domain-independent PCG iteration
counts observed in practice.
We apply this method to ground state computation
via inverse iteration for linear problems and via the
$a_u$ gradient flow for Gross--Pitaevskii energy in 3D, 
and also Hamiltonian simulation via the approximated qHOP and Magnus-2
splitting methods from 3D to 9D on a single NVIDIA GH200 GPU.
\end{abstract}

\begin{keywords}
tensor-product spectral method, GPU computing, Schr\"odinger
equation, Gross--Pitaevskii equation, Strang splitting, Yoshida splitting
\end{keywords}

\begin{AMS}
65M70, 65N35, 65F15, 81-08, 35Q55, 81Q05
\end{AMS}

\section{Introduction}
\label{sec:intro}

It is well known that the discrete Laplacian on Cartesian meshes
has a tensor-product structure, which can be used as a simple and
fast direct solver since the
1960s~\cite{lynch1964tensor,HAIDVOGEL1979167,shen1994,kwan2007efficient,ZhangShenLiu}.
For a general potential $V$, this structure is lost for the
Schr\"odinger operator $-\Delta + V$.
In this paper, we extend the tensor-product solver from the
Laplacian to the Schr\"odinger operator $-\Delta + V$ as follows.
For a separable potential $V_1(\boldsymbol{x}) = \sum_k f_k(x_k)$,
the operator $-\Delta + V_1$ can still be inverted or exponentiated via
per-axis eigendecomposition at cost $O(N^ {1+1/d})$ with $O(N)$
memory.
For non-separable potentials $V = V_1 + V_2$, the tensor-product
solver provides a preconditioner $(-\Delta + V_1)^{-1}$ for the preconditioned conjugate gradient (PCG) method
and a propagator for operator-splitting time integrators.
For bounded $V_2$, we prove that the preconditioned operator $(-\Delta + V_1)^{-1}(-\Delta+V_1+V_2)$
has a bounded condition number and a clustered spectrum with at most finitely many outliers,
independently of the mesh size (Theorem~\ref{thm:pcg}), and also
independently of the domain size when $V_1$ is a confining
potential (Theorem~\ref{cor:confining}).
This can explain mesh- and domain-independent PCG iteration counts observed in numerical tests for suitable potentials.
Similar to the simple GPU acceleration of the direct inversion of
the Laplacian in~\cite{ZhangShenLiu}, we show
that the Schr\"odinger operator $-\Delta + V$ can be easily
inverted on GPU
(Sections~\ref{sec:solver}--\ref{sec:gpe}), e.g., it takes less than 1 second for inverting $-\Delta + V_1$ for one billion DoFs on one Nvidia GPU such as A100 and GH200.

Fast inversion of the Schr\"odinger operator has several
applications. We first consider ground state
computation (Section~\ref{sec:gpe_app}): for separable potentials via shifted
inverse iteration, and for non-separable potentials
via shifted inverse iteration with PCG, preconditioned by
$(-\Delta + V_1)^{-1}$.
We also demonstrate that the defocusing Gross--Pitaevskii
ground state via the $a_u$ flow can be easily implemented in 3D
on a single GPU.

Another application is Hamiltonian simulation.
If $(\lambda_1, u_1)$ is the ground state eigenpair of
$H = -\Delta + V$, then
$\psi(x,t) = e^{-i\lambda_1 t}\,u_1(x)$ is an exact solution of the
time-dependent Schr\"odinger equation $i\partial_t \psi = H\psi$.
We use this stationary solution as a reference for testing
operator-splitting methods
(Sections~\ref{sec:hamiltonian}--\ref{sec:multibody}).

In the quantum computing literature, two recent splitting methods
have been considered for simulating
$i\partial_t \psi = (A + B)\psi$ in the interaction picture,
where $A$ has large spectral norm but can be fast-forwarded
(e.g., $A = -\Delta$ via QFT).
The \emph{quantum Highly Oscillatory Protocol}
(qHOP)~\cite{AnFangLin2022} approximates the first-order Magnus
expansion and achieves $O(\Delta t^2)$ superconvergence with an
error preconstant independent of problem size $n$~\cite{AnFangLin2022}.
The \emph{Magnus-2} algorithm~\cite{FangLiuSarkar2025} uses the
second-order Magnus expansion and achieves $O(\Delta t^4)$
superconvergence~\cite{FangLiuSarkar2025,BornsWeilFangZhang2026}.
On a quantum computer, the matrix exponential of a sum of
interaction-picture Hamiltonians is implemented via
linear combination of unitaries (LCU).
Classically, this cannot be computed directly at large $n$.
We implement an approximate version using a product formula
where each factor is computed exactly via tensor-product
propagations and a pointwise multiplication
(equation~\eqref{eq:qhop_product} and
Remark~\ref{rem:product_order} in
Section~\ref{sec:ham_separable}).
We test these methods for multi-body Coulomb Hamiltonians
in 4D, 6D, and 9D (Section~\ref{sec:multibody}).
The approximated versions of qHOP and Magnus-2 are tested in
multiple dimensions for two purposes.
First, these tests can serve as a verification of their theoretical superconvergence orders
in high dimensions.
Since the approximated versions introduce additional
product-formula errors beyond the original splitting error,
observing the predicted convergence orders despite these
additional errors provides supporting numerical evidence for
the original schemes. 
Second, the same product-formula implementation yields a family
of classical splitting methods with parameter $M$: $M = 1$ recovers the classical
Strang and Yoshida splittings, while $M \ge 3$ reduces the
error constant at the same convergence order, demonstrating
that qHOP and Magnus-2 can serve as improved classical
integrators compared to standard schemes.
Furthermore, as a classical numerical scheme, the splitting $A = -\Delta + V_1$, $B = V_2$
gives   smaller errors than
$A = -\Delta$, $B = V_1 + V_2$
(Tables~\ref{tab:splitting_nonsep}
and~\ref{tab:magnus2_nonsep}
in Section~\ref{sec:hamiltonian}).

In related work, the tensor-product approach to Schr\"odinger
equations has been explored by Caliari et al.~\cite{caliari2022},
whose $\mu$-mode
integrator exploits Kronecker structure for 3D problems at
${\sim}2\times 10^6$ DoFs.
GPU acceleration for  Schr\"odinger  simulation has also been explored
in~\cite{gainullin2015}.
The split-step spectral method was introduced in~\cite{bao2002}.
We refer to~\cite{lasser2020,jin2011} for reviews of
semiclassical methods and
to~\cite{jahnke2000,thalhammer2008,hochbruck2003} for splitting
error analysis and classical Magnus integrators. 
On the quantum computing side, gate complexity for
real-space simulation of Schr\"odinger
equations was established
in~\cite{ChildsLengLiLiuZhang2022}, Trotter error bounds for
interacting electrons were given
in~\cite{SuHuangCampbell2021,hahn2024}, and quantum resource
estimates for 2D Coulomb simulation are given
in~\cite{coulomb2026}.

All results in this paper can be reproduced via Python (JAX) code
available at \url{https://github.com/zhan1966/schrodinger}.
GPU cards are available to many researchers through programs such
as the NSF-funded ACCESS (Advanced Cyberinfrastructure Coordination
Ecosystem: Services \& Support)~\cite{ACCESS2023}, which provides
allocations on GPU clusters including NVIDIA A100 and GH200 nodes.
The direct inversion of the tensor-structured Laplacian on A100
was demonstrated
in~\cite{ZhangShenLiu,ChenLuLuZhang2024,HaoLeeZhang2025}.
This paper extends the same approach to the Schr\"odinger operator
$-\Delta + V$.
All computations in this paper require only a single GPU. 

The rest of the paper is organized as follows.
Section~\ref{sec:solver} describes the tensor-product spectral
solver.
Section~\ref{sec:gpe} develops the PCG preconditioner and
analyzes its mesh- and domain-independent condition number and spectral clustering.
Section~\ref{sec:gpe_app} presents ground state computation
via shifted inverse iteration and Gross--Pitaevskii gradient flows.
Sections~\ref{sec:hamiltonian} and~\ref{sec:multibody} present
Hamiltonian simulation results in 3D through 9D.
Some concluding remarks are given in Section~\ref{sec:conclusion}.
Appendix~\ref{app:hermite} presents the Hermite spectral
discretization on unbounded domains.

\section{Direct solver for $-\Delta + V$ with a separable potential}
\label{sec:solver}

The idea in this paper applies to any numerical method with a tensor-product
structure on rectangular domains.
For simplicity, throughout the paper we focus on the $Q^k$ spectral element
method (SEM), which is the high order finite element method with $Q^k$ basis and $(k+1)$-point Gauss-Lobatto quadrature. 
For $k=1$, it is the classical second order finite difference method.
For $k\geq 2$, $Q^k$ SEM can be proven to be a $(k+2)$-th order finite difference scheme in $\ell^2$-norm over quadrature point values 
~\cite{LiAppeloZhang2022}.
As an example of another suitable method, we also include a
Hermite spectral method on unbounded domains in
Appendix~\ref{app:hermite}.
See also~\cite{ZhangShenLiu} for using   $Q^k$ SEM in a Poisson solver.

\subsection{A direct solver via eigendecomposition for separable potentials}

We consider the elliptic PDE
\begin{equation}\label{eq:main}
  -\Delta u(\boldsymbol{x}) + V(\boldsymbol{x})u(\boldsymbol{x})
  = b(\boldsymbol{x}), \quad \boldsymbol{x} \in \Omega,
\end{equation}
on a rectangular domain $\Omega = [-L,L]^d$ (or on $\R^d$ via
Hermite functions in Appendix~\ref{app:hermite}).
For simplicity, we consider the homogeneous Dirichlet boundary
condition, and extensions of the scheme to homogeneous Neumann and periodic
boundary conditions are straightforward.
With the $Q^k$ spectral element method using Gauss--Lobatto
quadrature~\cite{LiAppeloZhang2022}, let $K_d = M_d^{-1}S_d$ denote the 1D
discrete Laplacian in direction $d \in \{x,y,z\}$,
with $M_d$ the mass matrix, $S_d$ the stiffness matrix,
and $I_d$ the identity matrix.
For $Q^1$ elements on a uniform mesh, $K_d$
is the well-known tridiagonal $(-1,2,-1)$ matrix from
second order finite difference.
Let $n$ be the number of interior grid points in each direction.
For any $d$-dimensional array $X$ of size $n \times \cdots \times n$,
let $\mathrm{vec}(X)$ denote the vector of size $n^d$ obtained by
reshaping all entries of $X$ into a column vector.
Then in 3D, the discretized system
takes the form (see~\cite{ZhangShenLiu,ChenLuLuZhang2024}):
\begin{equation}\label{eq:general_system}
  \bigl[I_z \otimes I_y \otimes K_x + I_z \otimes K_y \otimes I_x
  + K_z \otimes I_y \otimes I_x
  + \mathrm{diag}(\mathrm{vec}(V))\bigr]\,\mathrm{vec}(U)
  = \mathrm{vec}(B),
\end{equation}
where $V$ is the array of potential values $V(\boldsymbol{x}_i)$
at the grid points.
When the potential is separable,
$V(\boldsymbol{x}) = f_x(x) + f_y(y) + f_z(z)$,
the potential matrix decomposes as
$\mathrm{diag}(\mathrm{vec}(V))
= I_z \otimes I_y \otimes F_x + I_z \otimes F_y \otimes I_x
+ F_z \otimes I_y \otimes I_x$
with $F_d = \mathrm{diag}(f_d(x_1),\ldots,f_d(x_n))$, and
the system~\eqref{eq:general_system} reduces to the Kronecker
form
\begin{equation}\label{eq:tensor_system}
  (I_z \otimes I_y \otimes H_x
  + I_z \otimes H_y \otimes I_x
  + H_z \otimes I_y \otimes I_x)\,\mathrm{vec}(U)
  = \mathrm{vec}(B),
\end{equation}
where $H_d = K_d + F_d = M_d^{-1}S_d + F_d$ for each axis
$d \in \{x,y,z\}$.
Each $H_d$ can be diagonalized by 
\begin{equation}\label{eq:eigen_decomp}
  H_d = M_d^{-1/2}(M_d^{-1/2}S_d M_d^{-1/2} + F_d)M_d^{1/2}
  = T_d \Lambda_d T_d^{-1},
\end{equation}
where $T_d = M_d^{-1/2}Q_d$, $T_d^{-1} = Q_d^T M_d^{1/2}$,
and $Q_d \Lambda_d Q_d^T$ is the eigendecomposition of the symmetric
matrix $M_d^{-1/2}S_d M_d^{-1/2} + F_d$.

The solution is then computed in four steps:
\begin{enumerate}[nosep]
  \item  Off-line setup (one-time): compute the eigendecomposition
    $H_d = T_d \Lambda_d T_d^{-1}$ for each axis $d$.
    Since $H_d$ is an $n \times n$ matrix, this costs $O(n^3)=O(N)$.
  \item Forward transform: apply $T_z^{-1} \otimes T_y^{-1} \otimes T_x^{-1}$ to $B$.
  \item Solve the diagonal system: divide by $(\Lambda_x)_{ii} + (\Lambda_y)_{jj} + (\Lambda_z)_{kk}$.
  \item Backward transform: apply $T_z \otimes T_y \otimes T_x$.
\end{enumerate}
Steps 2 and 4 contain the dominant computational cost  $\mathcal{O}(N^{4/3})$ for $N = n^3$.
The memory cost is $\mathcal{O}(N)$.

The same factorization applies to the time-dependent
Schr\"odinger propagator $e^{-i(-\Delta + V_1)\Delta t}$.
Since $-\Delta + V_1 = T\,\Lambda\,T^{-1}$, where
$T = T_z \otimes T_y \otimes T_x$ and $\Lambda$ is the diagonal
tensor of eigenvalue sums
$\Lambda_{ijk} = (\Lambda_x)_{ii} + (\Lambda_y)_{jj} + (\Lambda_z)_{kk}$,
the matrix exponential is
\[
  e^{-i(-\Delta + V_1)\Delta t}\,\psi
  = T\,(e^{-i\Lambda\Delta t} \odot (T^{-1}\,\psi)),
\]
with $\odot$ denoting pointwise multiplication. The only change from the direct inversion is that the eigenvalue division in step~3 is replaced by multiplication
by $e^{-i\Lambda_{ijk}\Delta t}$. Steps 1, 2 and 4 are the same.
Thus, once the eigendecomposition is computed, both the direct
solve $(-\Delta + V_1)^{-1}b$ and the propagator
$e^{-i(-\Delta + V_1)\Delta t}\psi$ cost $O(N^ {1+1/d})$.

\subsection{Numerical tests}
\label{sec:sem_accuracy}

We demonstrate the solver on the equation
\begin{equation}\label{eq:test_problem}
  (-\Delta + V)\,u = f, \quad \boldsymbol{x} \in [-1,1]^d.
\end{equation}
For 3D, we use the non-isotropic separable potential
\begin{equation}\label{eq:test_potential}
  V(\boldsymbol{x}) = 1600\Bigl(\sin^2\bigl(\tfrac{\pi}{4}x\bigr)
  + \sin^2\bigl(\tfrac{\pi}{4}y\bigr)
  + \sin^2\bigl(\tfrac{\pi}{4}z\bigr)\Bigr)
  + (x^2 + 2y^2 + 3z^2),
\end{equation}
exact solution $u^* = \sin(\pi x)\sin(2\pi y)\sin(3\pi z)$,
and $f = (-\Delta + V)\,u^*$.
 For
 4D and 6D tests, we use the
harmonic trapping potential
$V(\boldsymbol{x}) =   \sum_{k=1}^{d} x_k^2$,
and the exact solution $u^* = \prod_{k=1}^{d} \sin(k\pi x_k)$.

All computations are implemented in Python with JAX, similar to the Python implementation in~\cite{ZhangShenLiu}.
Table~\ref{tab:accuracy_gh200} reports 
\emph{setup} time for  offline eigendecomposition and eigenvalue array
construction, and
\emph{solve} time for applying the direct solver,
which would be the per-iteration cost in an iterative method such as PCG in later sections.

In Table~\ref{tab:accuracy_gh200}, we also report different precision formats on a GH200 GPU
(96\,GB HBM3).
FP32 (single precision) stores each number in 32 bits, using half the memory of FP64.
TF32 (TensorFloat-32) uses 32-bit storage but performs matrix multiplications on GPU tensor cores at reduced precision
for fast matrix operations.  See~\cite{ZhangShenLiu} for more comparison of TF32 with FP32.
BF16 (bfloat16) uses only 16 bits per number, further
reducing memory and enabling larger grids on the same device, but with only
${\sim}3$-digit precision. 

\begin{table}[ht]\centering
\caption{Direct inversion of $(-\Delta + V)$ via $Q^k$ SEM on Nvidia GH200 96\,GB.
3D: potential~\eqref{eq:test_potential},
4D/6D: $V = \sum x_k^2$.
At matched $n$, solve times (in seconds) are identical across $Q^k$.
In FP64, all three dimensions reach ${\sim}10^9$ DoFs:
$1099^3 \approx 1.3\times10^9$,
$179^4 \approx 1.0\times10^9$,
$33^6 \approx 1.3\times10^9$.}
\label{tab:accuracy_gh200}
{\small
\begin{tabular}{@{}lclrrrrrl@{}}
\toprule
Precision & Dim & $Q^k$ & $N_{\text{cell}}$ & $n$ & DoFs & Setup (s) & Solve (s) & $\ell^2$ rel.\ err \\
\midrule
FP64 & 3D & $Q^{1}$  & 1100 & 1099 & $1099^3$ & 5.2 & 0.508 & $3.66\times10^{-6}$  \\
     &    & $Q^{2}$  &  550 & 1099 & $1099^3$ & 5.1 & 0.512 & $1.93\times10^{-9}$  \\
     &    & $Q^{10}$ &  110 & 1099 & $1099^3$ & 6.4 & 0.510 & $7.66\times10^{-13}$$^\dagger$ \\
\cmidrule{2-9}
     & {4D} & {$Q^{10}$} & {18} & {179} & {$179^4$} & {6.2} & {0.195} & {$1.29\times10^{-13}$$^\dagger$} \\
\cmidrule{2-9}
     & {6D} & {$Q^{2}$}  & {17} & {33}  & {$33^6$}  & {5.9} & {0.258} & {$2.95\times10^{-2}$$^\S$} \\
\midrule
FP32 & 3D & $Q^{2}$  &  500 & 999  & $999^3$  & 5.4 & 0.270 & $7.17\times10^{-7}$  \\
\cmidrule{2-9}
     & 4D & $Q^{10}$ &  15  & 149  & $149^4$  & 5.8 & 0.058 & $4.99\times10^{-7}$  \\
     &    & $Q^{10}$ &  20  & 199  & $199^4$  & 5.8 & 0.188 & $6.20\times10^{-7}$$^\ddagger$  \\
\cmidrule{2-9}
     & 6D & $Q^{10}$ &   3  &  29  & $29^6$   & 5.7 & 0.052 & $4.75\times10^{-4}$  \\
\midrule
TF32 & 3D & $Q^{2}$  &  500 & 999  & $999^3$  & 6.6 & 0.133 & $5.13\times10^{-4}$  \\
\midrule
BF16 & 3D & $Q^{2}$  &  500 & 999  & $999^3$  & 5.4 & 0.088 & $1.65\times10^{-2}$  \\
     &    & $Q^{2}$  &  800 & 1599 & $1599^3$ & 5.6 & 0.517 & $1.64\times10^{-2}$  \\
\bottomrule
\end{tabular}
}

\medskip\footnotesize
$\dagger$\,FP64 rounding floor.
$\ddagger$\,FP32 rounding floor.
$\S$\,Mesh too coarse for $Q^2$ to resolve  $ \prod_{k=1}^6 \sin(k\pi x_k)$.
\end{table}

\section{Efficient Inversion of $-\Delta + V$}
\label{sec:gpe}

When $V = V_1 + V_2$ contains a separable $V_1$ and a non-separable $V_2$, we use PCG with the 
preconditioner $(-\Delta + V_1)^{-1}$.

\subsection{Test potentials}\label{sec:test_potentials}

We consider two physically motivated potentials.

\noindent\emph{(a) Harmonic plus quartic potential}~\cite{antoine2017}:
\begin{equation}\label{eq:quartic}
  V(x,y,z) = 2(1{-}\alpha)(\gamma_x x^2 + \gamma_y y^2)
  + \tfrac{\kappa}{2}(x^2+y^2)^2 + \gamma_z z^2.
\end{equation}
We consider $\gamma_x = \gamma_y = 1$, $\gamma_z = 3$,
$\alpha = 1.4$, $\kappa = 0.3$,
and we split $V = V_1 + V_2$ with separable part
$V_1 = \gamma_x x^2 + \gamma_y y^2 + \gamma_z z^2$ and
non-separable remainder
\[
  V_2 = (2(1{-}\alpha){-}1)(\gamma_x x^2 {+} \gamma_y y^2)
  + \tfrac{\kappa}{2}(x^2{+}y^2)^2
  = -1.8(x^2{+}y^2) + 0.15(x^2{+}y^2)^2.
\]
The quartic coupling $(x^2+y^2)^2$ makes $V_2$ genuinely non-separable. 

\smallskip
\noindent\emph{(b) Harmonic plus stirrer potential}~ \cite{AntoineDuboscq2015,JacksonMcCannAdams1998, bao2004computing}:
\begin{equation}\label{eq:stirrer}
  V(x,y,z) = \underbrace{\gamma_x^2 x^2 + \gamma_y^2 y^2 + \gamma_z^2 z^2}_{V_1}
  + \underbrace{2w_0\, e^{-\delta\left((x - r_0)^2 + y^2\right)}}_{V_2},
\end{equation}
with $\gamma_x = 1$, $\gamma_y = 1$, $\gamma_z = 2$, $w_0 = 4$, $\delta = 1$, $r_0 = 1$.
This potential includes a harmonic trap with a localized Gaussian
stirrer.

\subsection{PCG for $-\Delta + V_1 + V_2$ (non-separable case)}
\label{sec:pcg_inv}

When $V = V_1 + V_2$ contains a non-separable component $V_2$,
the system $(-\Delta + V_1 + V_2)\,u = f$ cannot be solved directly
by the tensor-product solver.
We consider the preconditioned conjugate gradient (PCG) method.
Antoine et al.~\cite{antoine2017} showed that PCG
with the combined preconditioner
$P_C = V^{-1/2}(-\Delta)^{-1}V^{-1/2}$
achieves grid-independent iteration counts.
In this subsection, we show that the preconditioner $  P = (-\Delta + V_1)^{-1}$ is stronger.
It is applied by the tensor-product direct solver
with a cost $O(N^{4/3})$ in 3D. 

\begin{example}[PCG on bounded domain]
\label{ex:pcg}
We consider the quartic~\eqref{eq:quartic} and
stirrer~\eqref{eq:stirrer} potentials on $[-8,8]^3$ with
$Q^6$ SEM, Dirichlet BC, a random right-hand side $f$,
FP32 arithmetic, on a GH200 GPU. In all FP32 PCG tests reported here, the offline eigendecomposition
of $-\Delta + V_1$ is computed in FP64 and only then cast to lower
precision for the online PCG iteration.
Figure~\ref{fig:pcg_convergence} shows the convergence at
two grid sizes ($197^3$ and $701^3$ DoFs).
Both preconditioners exhibit grid-independent iteration counts.
The proposed preconditioner converges in 33--38 iterations
(quartic) and 5--6 iterations (stirrer), compared to 214--255
iterations with $(-\Delta)^{-1}$ for the quartic
potential 
and ${\sim}375$ iterations with $P_C = V^{-\frac12}(-\Delta)^{-1}V^{-\frac12}$ for the stirrer potential.
\end{example}

\begin{figure}[ht]\centering
\includegraphics[width=0.495\textwidth]{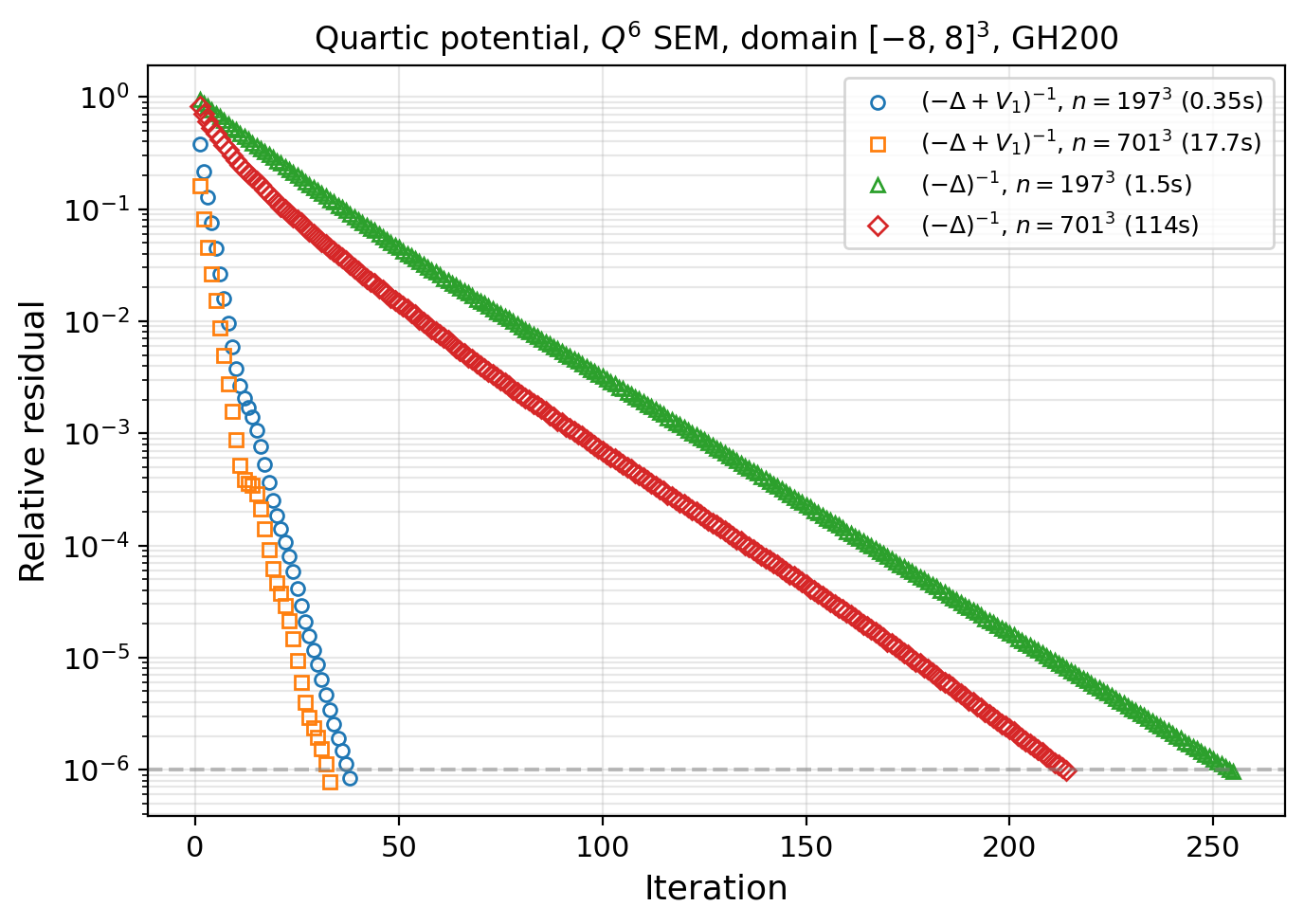}\hfill
\includegraphics[width=0.495\textwidth]{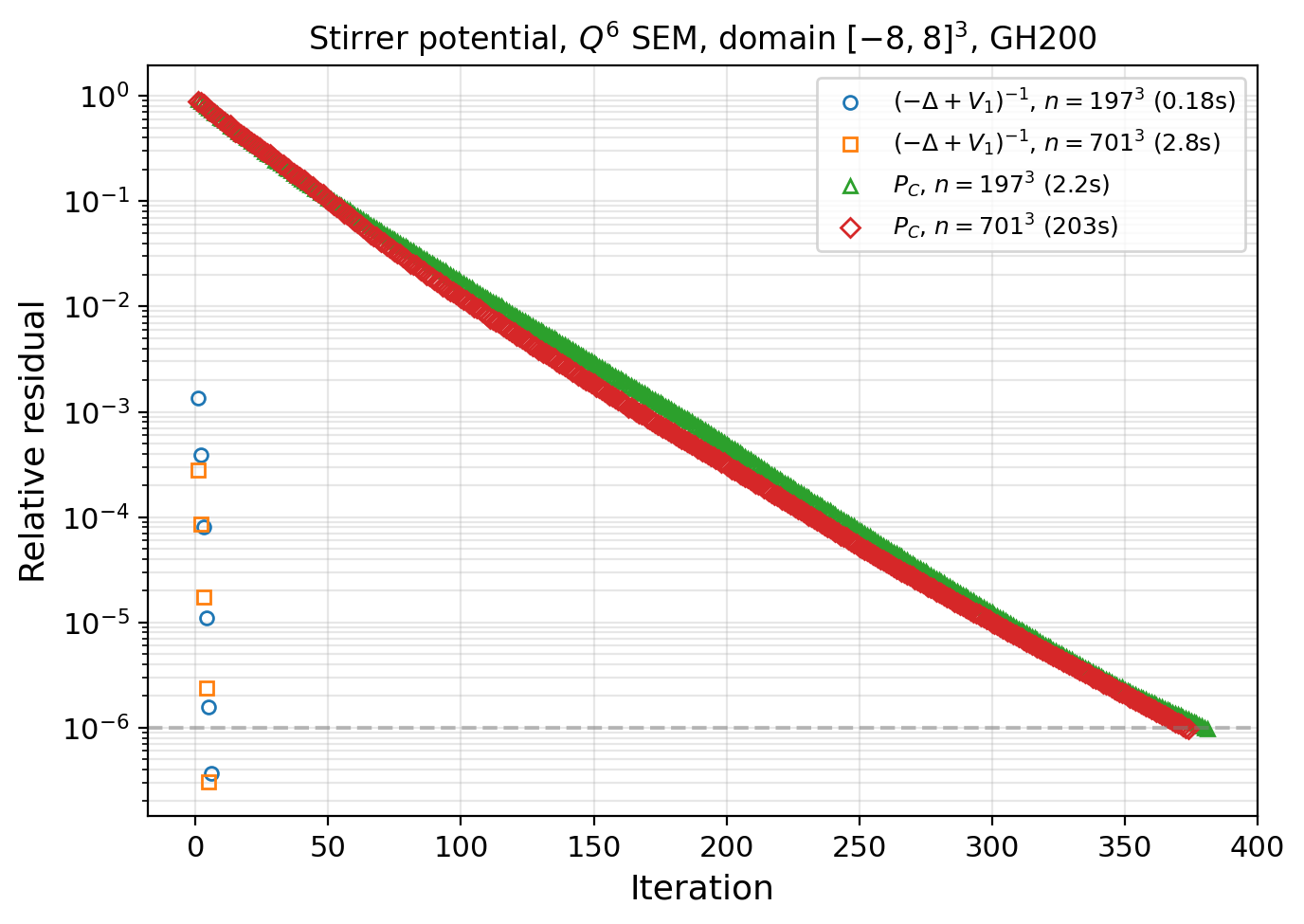}\\[4pt]
\includegraphics[width=0.495\textwidth]{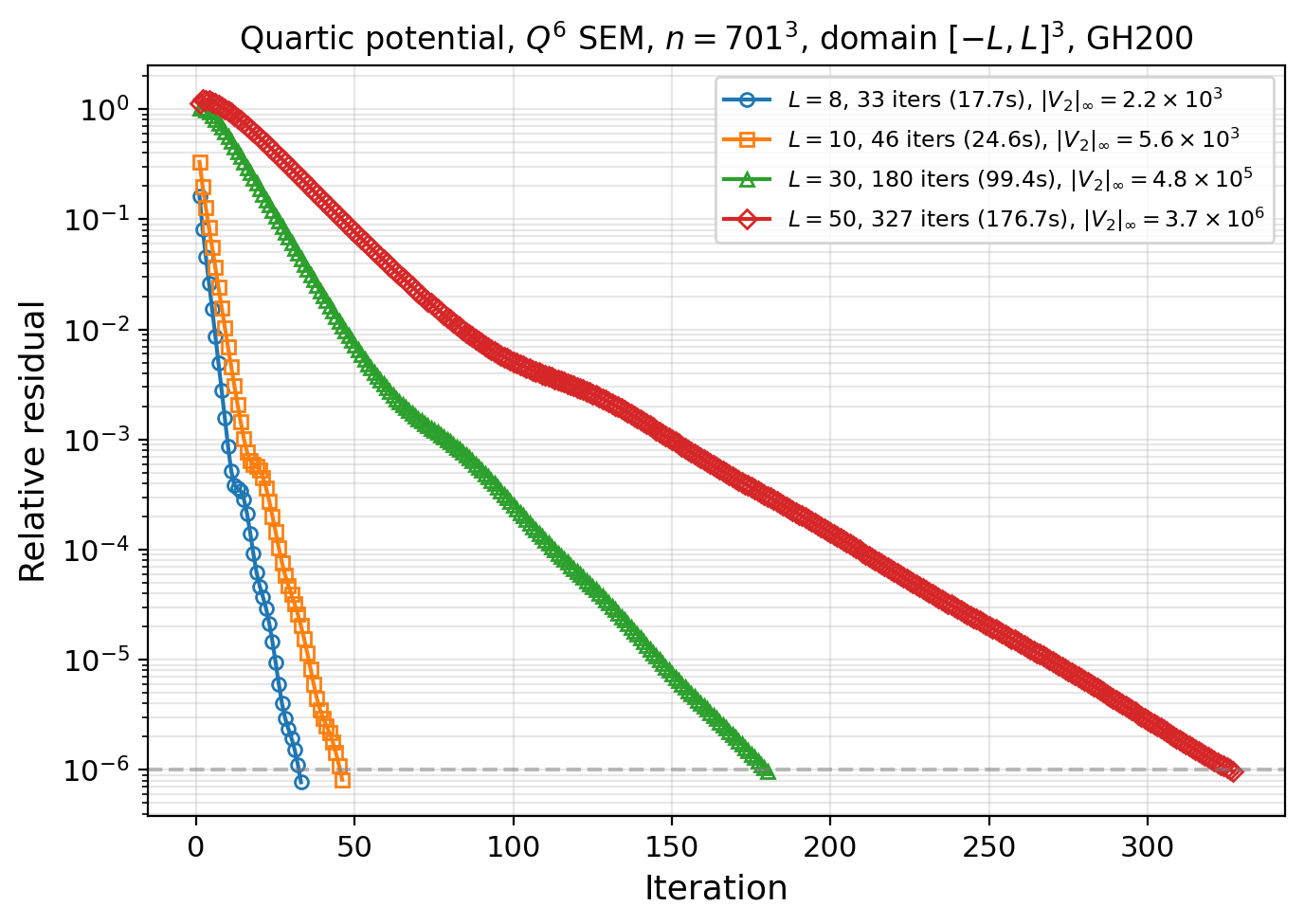}\hfill
\includegraphics[width=0.495\textwidth]{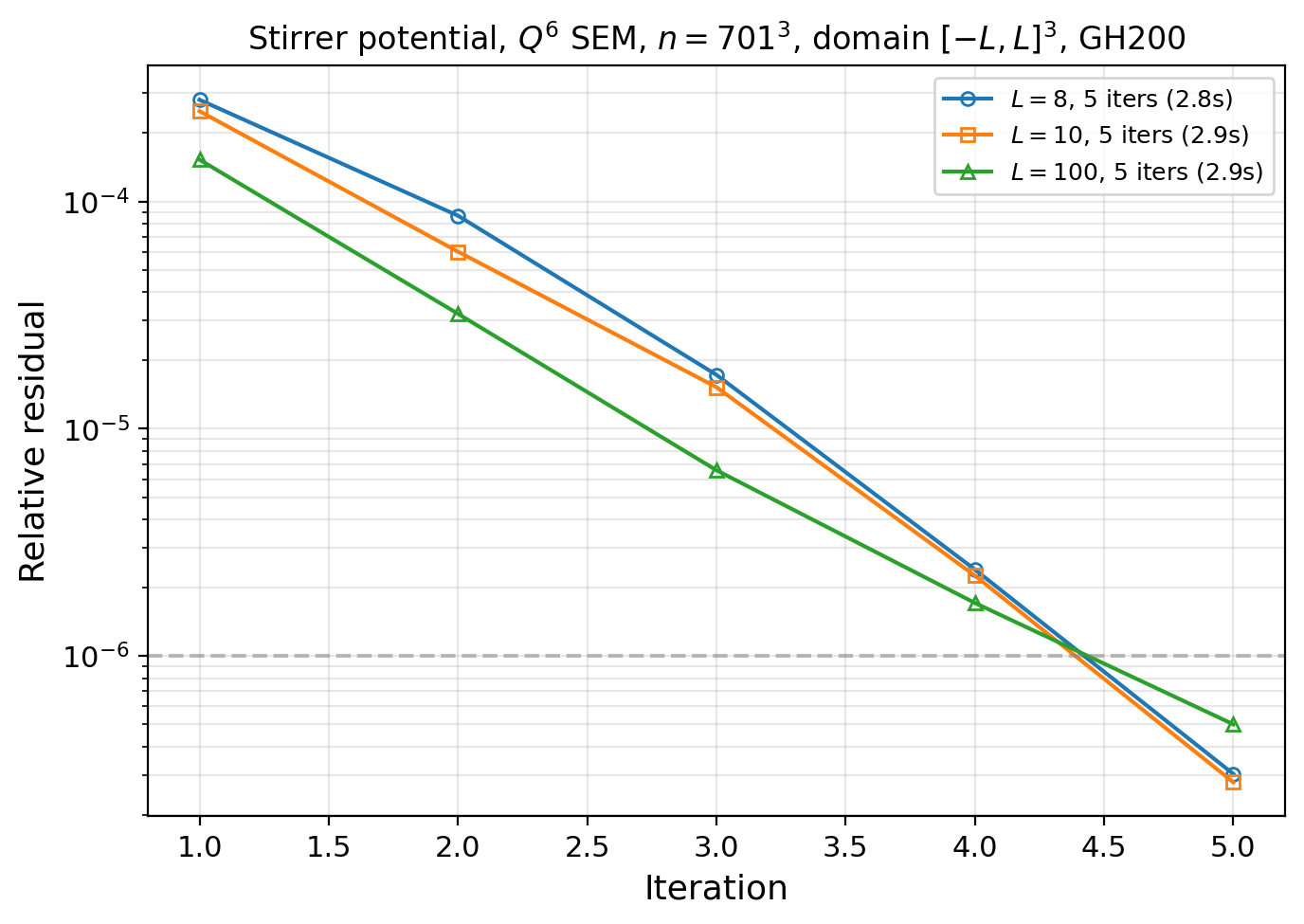}
\caption{PCG convergence (relative residual vs.\ iteration) for $(-\Delta + V_1 + V_2)\,u = f$,
$Q^6$ SEM, FP32, GH200.
\emph{Top row:} quartic potential~\eqref{eq:quartic} (left) and
stirrer potential~\eqref{eq:stirrer} (right), both at $n=701^3$,
comparing the proposed preconditioner $(-\Delta+V_1)^{-1}$ with the
$(-\Delta)^{-1}$ or $P_C$.
\emph{Bottom row:} domain dependence of $(-\Delta+V_1)^{-1}$ for the
quartic potential (left, $L=8,10,30,50$) and the stirrer potential
(right, $L=8,10,100$).
Total GPU time for each case is shown in the legend.}
\label{fig:pcg_convergence}
\end{figure}

Figure~\ref{fig:pcg_convergence} also shows that the 
performance of the preconditioner $(-\Delta + V_1)^{-1}$ is independent of both the domain size $L$ and the
mesh size $h$ for the stirrer potential, while for the quartic
potential, the iteration count grows with $L$.

In the rest of this section, we explain this behavior by analyzing
the spectrum of the preconditioned operator, which involves standard spectral theory
for conforming finite element discretizations
(such as $Q^k$ SEM) under two assumptions on
$V_1$ and $V_2$.

\begin{description}[nosep]
\item[\textbf{Assumption~I}] (bounded domain):
  $V_1 \in L^\infty(\Omega)$ with $V_1 \ge 0$, and
  $V_2 \in L^\infty(\Omega)$.
\item[\textbf{Assumption~II}] (confining potential):
  $V_1 \ge 0$ with $V_1(x) \to \infty$ as $|x|\to\infty$ and
  $V_1 \in L^\infty_{\mathrm{loc}}(\R^d)$, and
  $V_2 \in L^\infty(\R^d)$.
\end{description}
Assumption~II is stronger since it implies
Assumption~I on any bounded domain $\Omega$.

\begin{definition}[$A$-inner product]
\label{def:A_inner}
For a bounded domain  $\Omega \subset \R^d$, under Assumption~I, 
define the inner product for  $H^1_0(\Omega)$:
\[
  \langle u,v\rangle_A
  = \int_\Omega(\nabla u \cdot \nabla v + V_1\,u\,v)\,\mathrm{d}\boldsymbol{x}.
\]
Since $V_1 \ge 0$ and $V_1 \in L^\infty(\Omega)$,
the Poincar\'e inequality gives
$(1+C_P^2)^{-1}\|u\|_{H^1}^2 \le \|u\|_A^2
\le \max(1,\|V_1\|_{L^\infty})\|u\|_{H^1}^2$,
so $\|\cdot\|_A$ is equivalent to the standard $H^1$ norm.
\end{definition}

\begin{definition}[Perturbation operators and Riesz map]
\label{def:operators}
Under the setting of Definition~\ref{def:A_inner},
let $X_h \subset H^1_0(\Omega)$ be a conforming finite element
subspace with mesh parameter $h > 0$.
Define:
\begin{enumerate}[label=\upshape(\alph*)]
\item The \emph{Riesz map}
$R : L^2(\Omega) \to H^1_0(\Omega)$ by
$\langle Rf, v\rangle_A = (f,v)_{L^2}$
for all $v \in H^1_0(\Omega)$.
\item The \emph{continuous perturbation operator}
$K : H^1_0(\Omega) \to H^1_0(\Omega)$ by
\begin{equation}\label{eq:K_def}
  \langle Ku, v \rangle_A
  = \int_\Omega V_2\,u\,v\,\mathrm{d}\boldsymbol{x},
  \quad \forall\, v \in H^1_0(\Omega).
\end{equation}
\item  
$K_h : X_h \to X_h$ by
$\langle K_h u_h, v_h \rangle_A
= \int_\Omega
V_2\,u_h\,v_h\,\mathrm{d}\boldsymbol{x}$
for all $v_h \in X_h$.
\item The \emph{projection}
$\Pi_h : H^1_0(\Omega) \to X_h$ by
$\langle \Pi_h u, v_h \rangle_A = \langle u, v_h \rangle_A$
for all $v_h \in X_h$.
\item The \emph{Galerkin discretizations}
$-\Delta_h + V_{1,h}$ and
$-\Delta_h + V_{1,h} + V_{2,h}$ on $X_h$:
\[
  ((-\Delta_h + V_{1,h})\,u_h,\, v_h)_{L^2}
  := \langle u_h, v_h \rangle_A,
\]
\[
  ((-\Delta_h + V_{1,h} + V_{2,h})\,u_h,\, v_h)_{L^2}
  := \langle u_h, v_h \rangle_A
  + \int_\Omega V_2\,u_h\,v_h
  \,\mathrm{d}\boldsymbol{x},
\]
for all $u_h, v_h \in X_h$.
Here $V_{2,h}$ denotes the restriction of $V_2$ to $X_h$
via the $L^2$ inner product, i.e.,
$(V_{2,h}\,u_h, v_h)_{L^2}
= \int_\Omega V_2\,u_h\,v_h
\,\mathrm{d}\boldsymbol{x}$.
\end{enumerate}
\end{definition}

The operators in Definition~\ref{def:operators} correspond directly
to the components of the preconditioned system:
$R$ represents  $(-\Delta + V_1)^{-1}$,
$K$ represents $(-\Delta + V_1)^{-1} V_2$,
and $K_h$ is the discrete version of $K$.
We now establish their key properties.

\begin{lemma}
\label{lem:operator_framework}
Under the setting of Definition~\ref{def:operators}:
\begin{enumerate}[label=\upshape(\alph*)]
\item $K = R \circ (V_2\,\cdot) \circ \iota$, where
$\iota : H^1_0(\Omega) \hookrightarrow L^2(\Omega)$
is the natural embedding.
In particular, $K$ is compact and self-adjoint on
$(H^1_0(\Omega), \langle\cdot,\cdot\rangle_A)$.
\item $K_h u_h = \Pi_h(K u_h)$ for all $u_h \in X_h$.
\item $K = (-\Delta + V_1)^{-1} V_2$ and
$K_h = (-\Delta_h + V_{1,h})^{-1} V_{2,h}$.
\end{enumerate}
\end{lemma}

\begin{proof}
(a) First, the multiplication
$u \mapsto V_2 u$ is bounded on $L^2$, since $V_2 \in L^\infty(\Omega)$. Second, $R$ is a bounded operator, since it represents $(-\Delta+V_1)^{-1}$.
The embedding $\iota$ is compact by the
Rellich--Kondrachov theorem, so
$K = R \circ (V_2\,\cdot) \circ \iota$ is compact.
For self-adjointness,
\eqref{eq:K_def} gives
$\langle Ku, v\rangle_A = \int V_2\,uv$
and  
$\langle Kv, u\rangle_A = \int V_2\,vu$.
Since $\langle\cdot,\cdot\rangle_A$ is symmetric,
$\langle Ku, v\rangle_A = \langle u, Kv\rangle_A$.

\medskip\noindent
(b)
For any $u_h, v_h \in X_h$,
\[
  \langle K_h u_h, v_h \rangle_A
  = \int V_2\,u_h\,v_h
  = \langle K u_h, v_h \rangle_A
  = \langle \Pi_h(K u_h), v_h \rangle_A,
\]
so $K_h u_h = \Pi_h(K u_h)$ for all $u_h \in X_h$,
i.e., $K_h = \Pi_h K|_{X_h}$.

\medskip\noindent
(c) By definition,
$\langle Ku, v \rangle_A = \int V_2\,u\,v
= (V_2 u, v)_{L^2}
= \langle R(V_2 u), v \rangle_A$,
so $Ku = R(V_2 u)$, i.e.,
$K = (-\Delta + V_1)^{-1} V_2$.
Similarly, $K_h u_h$ is the unique element of $X_h$
satisfying
$\langle K_h u_h, v_h \rangle_A
= \int V_2\,u_h\,v_h$
for all $v_h \in X_h$, which is precisely
$(-\Delta_h + V_{1,h})^{-1}(V_{2,h}\,u_h)$.
\end{proof}

\begin{theorem}[Spectral convergence of the preconditioned operator]
\label{thm:pcg}
For a bounded domain  $\Omega \subset \R^d$, under Assumption~I, 
let $T = (-\Delta + V_1)^{-1}(-\Delta + V_1 + V_2) = I + K$
and
$T_h = (-\Delta_h + V_{1,h})^{-1}(-\Delta_h + V_{1,h} + V_{2,h})
= I_h + K_h$,
where $I_h$ denotes the identity on $X_h$,
be the continuous and discrete preconditioned operators,
and let
$\widetilde{K}_h = \Pi_h K \Pi_h : H^1_0(\Omega) \to
H^1_0(\Omega)$.
Then:
\begin{enumerate}[label=\upshape(\roman*)]
\item \textbf{Norm convergence.}
$\|\widetilde{K}_h - K\| \to 0$ as $h \to 0$
in the operator norm on $(H^1_0(\Omega), \|\cdot\|_A)$.
\item \textbf{Spectral convergence and clustering.}
For every $\epsilon > 0$, $T_h$ has at most
$M(\epsilon)$ eigenvalues outside $(1{-}\epsilon,1{+}\epsilon)$,
where $M(\epsilon)$ is independent of $h$.
Each nonzero eigenvalue $\mu$ of $K$ with
multiplicity $m$ is approximated by exactly $m$ eigenvalues
of $K_h$ converging to $\mu$ as $h \to 0$.
\item \textbf{Uniform condition number.}
If $-\Delta + V_1 + V_2 > 0$, then
$\kappa(T_h) \le \kappa(T)$ for every $h$,
and $\kappa(T_h) \to \kappa(T)$ as $h \to 0$.
\end{enumerate}
\end{theorem}

\begin{proof}
Since $\Pi_h$ maps $H^1_0(\Omega)$ onto $X_h$,
$\widetilde{K}_h = \Pi_h K \Pi_h$ maps $X_h^\perp$ to zero
and restricts to $K_h$ on $X_h$
(by Lemma~\ref{lem:operator_framework}(b),
$\Pi_h K u_h = K_h u_h$ for $u_h \in X_h$).
Therefore the nonzero eigenvalues of $\widetilde{K}_h$
are exactly those of $K_h$:
if $\widetilde{K}_h f = \mu f$ with $\mu \ne 0$,
then $f = \mu^{-1}\widetilde{K}_h f \in X_h$,
so $f$ is an eigenfunction of $K_h$.

\medskip\noindent
\emph{Part~\textup{(i)}: Norm convergence.}
Write $K - \widetilde{K}_h
= (I - \Pi_h)K + \Pi_h K(I - \Pi_h)$.
Since $K$ and $\Pi_h$ are both self-adjoint in
$\langle\cdot,\cdot\rangle_A$,
the adjoint of $\Pi_h K(I - \Pi_h)$ is
$(I - \Pi_h)K\Pi_h$.
Because $\|B\| = \|B^*\|$ for any bounded operator $B$ and its adjoint $B^*$
on a Hilbert space,
$\|\Pi_h K(I - \Pi_h)\|
= \|(I - \Pi_h)K \Pi_h\|
\le \|(I-\Pi_h)K\|\,\|\Pi_h\|
= \|(I-\Pi_h)K\|$, thus
\begin{equation}\label{eq:norm_split}
\|\widetilde{K}_h - K\|
\le 2\,\|(I - \Pi_h)K\|.
\end{equation}
It remains to show $\|(I - \Pi_h)K\| \to 0$.
Since $K$ is compact,
the image of the closed unit ball
$\overline{B}_1 = \{f \in H^1_0(\Omega) : \|f\|_A \le 1\}$
under $K$ is precompact in
$(H^1_0(\Omega), \|\cdot\|_A)$.
Consider any fixed $\epsilon > 0$, and 
cover $K(\overline{B}_1)$ by finitely many $\epsilon$-balls
centered at $g_1, \ldots, g_M\in H_0^1(\Omega)$.
Since $C^\infty_c(\Omega)$ is dense in $H^1_0(\Omega)$
and conforming finite element spaces approximate
smooth functions as $h \to 0$ (by polynomial interpolation),
$\bigcup_h X_h$ is dense in $H^1_0(\Omega)$, so
$(I - \Pi_h)g_j \to 0$ in $\|\cdot\|_A$
for each $j$, thus there exists $h_0$ such that
$\|(I - \Pi_h)g_j\|_A < \epsilon$ for all $j$ and all
$h < h_0$.
For any $f$ with $\|f\|_A \le 1$, choose $g_j$ with
$\|Kf - g_j\|_A < \epsilon$.
Since $\|I - \Pi_h\| \le 1$
(orthogonal projection),
\[
  \|(I - \Pi_h)Kf\|_A
  \le \underbrace{\|(I - \Pi_h)(Kf - g_j)\|_A}_{\le\,\epsilon}
  + \underbrace{\|(I - \Pi_h)g_j\|_A}_{<\,\epsilon}
  < 2\epsilon.
\]
Hence $\|(I - \Pi_h)K\| \le 2\epsilon$ for $h < h_0$.
Since $\epsilon$ was arbitrary,
$\|(I - \Pi_h)K\| \to 0$, and
by~\eqref{eq:norm_split},
$\|\widetilde{K}_h - K\| \to 0$.

\medskip\noindent
\emph{Part~\textup{(ii)}:}
\emph{Clustering.}
Since $\Pi_h$ is an orthogonal projection  in
$\langle\cdot,\cdot\rangle_A$, its operator norm is $\|\Pi_h\| = 1$, so the
 singular values $s_n$ satisfy
$s_n(\widetilde{K}_h) = s_n(\Pi_h K \Pi_h) \le s_n(K)$.
For self-adjoint operators, singular values $s$  equal absolute
values of eigenvalues $\mu$, so $|\mu_n(K_h)| \le |\mu_n(K)|$
for each $n$, where eigenvalues are ordered by
decreasing absolute value.
Since $K$ is compact and self-adjoint, its eigenvalues
accumulate only at $0$ ~\cite[Theorem~VI.16]{reedsimon1972},
so for every $\epsilon > 0$, at most
$M(\epsilon) < \infty$ of them satisfy
$|\mu_n(K)| \ge \epsilon$.
Since  $|\mu_n(K_h)| \le |\mu_n(K)|$, the same bound holds for $K_h$, independently of $h$.
Hence $T_h = I_h + K_h$ has at most $M(\epsilon)$
eigenvalues outside $(1{-}\epsilon, 1{+}\epsilon)$.

\emph{Convergence.}
Since $\widetilde{K}_h \to K$ in operator norm and $K$ and $\widetilde{K}_h$ are compact
and self-adjoint, Osborn's spectral approximation
theory~\cite{osborn1975} applies directly.
By the spectral convergence results in
\cite[Section~3]{osborn1975},
for each  nonzero eigenvalue $\mu$ of $K$
with multiplicity $m$, there exist exactly $m$ eigenvalues of
$\widetilde{K}_h$ converging to $\mu$.
Since the nonzero eigenvalues of $\widetilde{K}_h$ are exactly
those of $K_h$,
the same convergence holds for the eigenvalues of $K_h$.

\medskip\noindent
\emph{Part~\textup{(iii)}: Uniform condition number.}
For any $u_h \in X_h$ with $\|u_h\|_A = 1$,
\[
  \langle K_h u_h, u_h\rangle_A
  = \int_\Omega V_2\,|u_h|^2\,\mathrm{d}\boldsymbol{x}
  = \langle K u_h, u_h\rangle_A,
\]
since $u_h \in X_h \subset H^1_0(\Omega)$.
By the Courant--Fischer min-max principle,
$\mu_{\max}(K_h) = \max_{u_h \in X_h, \|u_h\|_A=1}
\langle K u_h, u_h\rangle_A$
and $\mu_{\max}(K) = \sup_{u \in H^1_0, \|u\|_A=1}
\langle K u, u\rangle_A$.
Since $X_h \subset H^1_0(\Omega)$, the maximum over the
smaller set $X_h$ cannot exceed the supremum over
$H^1_0(\Omega)$.
The same reasoning applies to $\mu_{\min}$.
Therefore:
\[
  \mu_{\min}(K) \le \mu_{\min}(K_h)
  \quad\text{and}\quad
  \mu_{\max}(K_h) \le \mu_{\max}(K).
\]
Since $T_h = I_h + K_h$ and $T = I + K$, this gives
$\mu_{\min}(T) \le \mu_{\min}(T_h)$ and
$\mu_{\max}(T_h) \le \mu_{\max}(T)$.
If $T > 0$, then $\mu_{\min}(T) > 0$, so
\[
  \kappa(T_h)
  = \frac{\mu_{\max}(T_h)}{\mu_{\min}(T_h)}
  \le \frac{\mu_{\max}(T)}{\mu_{\min}(T)}
  = \kappa(T)
  \quad\text{for every } h.
\]
Since $\mu_{\max}(K)$ and
$\mu_{\min}(K)$ (if nonzero) are isolated eigenvalues of $K$,
Part~(ii) gives eigenvalues of $K_h$ converging to them.
Combined with $\mu_{\max}(K_h) \le \mu_{\max}(K)$, this gives
$\mu_{\max}(K_h) \to \mu_{\max}(K)$.
Similarly $\mu_{\min}(K_h) \to \mu_{\min}(K)$.
Hence $\kappa(T_h) \to \kappa(T)$.
\end{proof}

\begin{theorem}[Domain-independent clustering for confining potentials]
\label{cor:confining}
Under Assumption~II,
consider the family of domains $\Omega_L = [-L,L]^d$ with
conforming finite element subspaces
$X_h \subset H^1_0(\Omega_L)$, and the preconditioned operator
$T_{h,L}$ defined as in Lemma~\ref{lem:operator_framework},
then for every $\epsilon > 0$, there exist at most $M(\epsilon)$
eigenvalues of $T_{h,L}$ outside $(1{-}\epsilon, 1{+}\epsilon)$,
where $M(\epsilon) < \infty$ is independent of both $h$ and $L$.
Moreover, if $-\Delta + V_1 + V_2 > 0$ on $\R^d$, then
$\kappa(T_{h,L})$ is bounded independently of both $h$ and $L$.

\end{theorem}

\begin{proof}
Since $V_1$ is confining, the Schr\"odinger operator
$-\Delta + V_1$ on $\R^d$ has compact resolvent and therefore
purely discrete spectrum
$\Lambda_1 \le \Lambda_2 \le \cdots \to \infty$ ~\cite[Theorem~XIII.16]{reedsimon1978}.
By the min-max principle, the $n$-th eigenvalue
of $-\Delta + V_1$ on $\Omega_L$ satisfies
\[
  \lambda_n(-\Delta_L+V_1)
  = \min_{\substack{\dim S = n \\ S \subset H^1_0(\Omega_L)}}
  \;\max_{u \in S,\, \|u\|_{L^2}=1}
  \|u\|_A^2.
\]
Restricting to
$H^1_0(\Omega_L) \subset H^1(\R^d)$ reduces the set of
candidate subspaces for the outer min, so the eigenvalues
can only increase:
$\lambda_n(-\Delta_L + V_1) \ge \Lambda_n$ for all $n$ and $L$.
The operator norm of the mapping $(-\Delta_L + V_1)^{-1}: L^2(\Omega_L) \to L^2(\Omega_L)$ satisfies
\begin{equation}\label{eq:resolvent_bound}
  \|(-\Delta_L + V_1)^{-1}\|_{L^2 \to L^2}
  = \frac{1}{\lambda_1(-\Delta_L + V_1)}
  \le \frac{1}{\Lambda_1},
  \quad \text{independently of } L.
\end{equation}
By contrast, for $\Omega_L=[-L,L]^d$,
$\|(-\Delta_L)^{-1}\|_{L^2 \to L^2} = 1/\lambda_1(-\Delta_L)
\sim L^2/\pi^2 \to \infty$.

Recall that Lemma~\ref{lem:operator_framework} gives
$K_L = R \circ (V_2\,\cdot) \circ \iota$, where
$V_2\,\cdot$ is bounded with norm $\|V_2\|_{L^\infty}$.
By the definition of $R$ with test function $Rf$,
$\|Rf\|_A^2 = (f, Rf)_{L^2}
\le \|f\|_{L^2}\,\|Rf\|_{L^2}
\le \|f\|_{L^2}\cdot\Lambda_1^{-1}\|f\|_{L^2}$
by~\eqref{eq:resolvent_bound}, so
$\|Rf\|_A \le \Lambda_1^{-1/2}\|f\|_{L^2}$. Thus
$R$ is bounded as a map $L^2(\Omega_L) \to
(H^1_0(\Omega_L), \|\cdot\|_A)$ with norm
$\le 1/\sqrt{\Lambda_1}$.

The singular values of the compact embedding $\iota : (H^1_0(\Omega_L), \|\cdot\|_A)
\hookrightarrow L^2(\Omega_L)$ are determined as follows.
The adjoint $\iota^* : L^2(\Omega_L) \to H^1_0(\Omega_L)$
satisfies $\langle \iota^* f, v\rangle_A = (f,v)_{L^2}$ for all
$v \in H^1_0(\Omega_L)$, so $\iota^* f$ solves
$(-\Delta + V_1)w = f$ weakly, i.e.,
$\iota^* = (-\Delta_L + V_1)^{-1}$ as a map
$L^2 \to H^1_0$.
Therefore $\iota^*\iota : H^1_0 \to H^1_0$ has eigenvalues
$1/\lambda_n(-\Delta_L + V_1)$, and the singular values of
$\iota$ are
$s_n(\iota) = 1/\sqrt{\lambda_n(-\Delta_L + V_1)}\leq 1/\sqrt{\Lambda_n}$.
Since $K_L$ is self-adjoint, its singular values equal the
absolute values of its eigenvalues, and the multiplicative
singular value inequality gives
\[
  |\mu_n(K_L)|
  \le \frac{\|V_2\|_{L^\infty(\R^d)}}
  {\sqrt{\Lambda_1\,\Lambda_n}}.
\]
Therefore $|\mu_n(K_L)| \ge \epsilon$ requires
$\Lambda_n \le \|V_2\|_{L^\infty}^2/(\Lambda_1\,\epsilon^2)$,
which holds for at most $M(\epsilon)$ values of $n$,
independently of $L$.
This bound transfers to the discrete operators.
Define $\widetilde{K}_{h,L} = \Pi_h K_L \Pi_h$
as in Theorem~\ref{thm:pcg}.
Since $\|\Pi_h\| = 1$ (orthogonal projection), the singular value
submultiplicativity gives
$s_n(\Pi_h K_L \Pi_h) \le s_n(K_L)$.
For self-adjoint operators $|\mu_n| = s_n$, so
$|\mu_n(K_{h,L})| \le |\mu_n(K_L)|$ for each $n$.
Hence $K_{h,L}$ also has at most $M(\epsilon)$ eigenvalues
outside $[-\epsilon,\epsilon]$, independently of both $h$ and $L$.
 
For the condition number bound,
Theorem~\ref{thm:pcg}(iii) gives
$\kappa(T_{h,L}) \le \kappa(T_L)$ for each $L$.
It remains to show $\kappa(T_L)$ is bounded independently
of $L$.
Let $K_{\R^d} = (-\Delta + V_1)^{-1}V_2$ on $\R^d$
(well-defined since $V_1$ is confining).
Any $u \in H^1_0(\Omega_L)$, extended by zero to
$\tilde{u} \in H^1(\R^d)$, satisfies
$\|\tilde{u}\|_A = \|u\|_A$ and
$\langle K_{\R^d}\tilde{u},\tilde{u}\rangle_A
= \int V_2|\tilde{u}|^2 = \int_{\Omega_L} V_2|u|^2
= \langle K_L u, u\rangle_A$.
Thus the Rayleigh quotients of $K_L$ over $H^1_0(\Omega_L)$
are a subset of those of $K_{\R^d}$ over $H^1(\R^d)$, giving
$\mu_{\min}(K_L) \ge \mu_{\min}(K_{\R^d})$ and
$\mu_{\max}(K_L) \le \mu_{\max}(K_{\R^d})$.
Therefore $\kappa(T_L) \le \kappa(T_{\R^d})$ for all $L$,
and $\kappa(T_{h,L}) \le \kappa(T_{\R^d})$
for all $h$ and $L$.
\end{proof}

\begin{remark}[Grid-independent PCG convergence]
\label{rem:pcg_discussion}
For the stirrer potential~\eqref{eq:stirrer}, $V_1$ is
confining and $V_2 \in L^\infty(\R^d)$, so
Theorem~\ref{cor:confining} guarantees that the eigenvalues of
$T_{h,L}$ cluster near $1$ with at most $M(\epsilon)$ outliers, and $\kappa(T_{h,L})$ is bounded, 
independently of both $h$ and $L$, which explains
mesh- and domain-independent PCG iteration counts. In particular,
 CG effectively ignores finitely many outlier eigenvalues ~\cite{vanderSluis1986},
which can explain the fast convergence (5--6 iterations)
observed in Example~\ref{ex:pcg}
(Figure~\ref{fig:pcg_convergence}, bottom right),
as well as the 5 iterations on $\R^3$ with the Hermite spectral
method (Figure~\ref{fig:pcg_hermite} in
Appendix~\ref{app:hermite}).
For the quartic potential~\eqref{eq:quartic},
$V_2 \notin L^\infty(\R^d)$, so only Theorem~\ref{thm:pcg}
applies.
For each fixed $L$,
$\kappa(T_h) \le \kappa(T)$ ensures mesh-independent
iteration counts
(Figure~\ref{fig:pcg_convergence}, top left),
but $\kappa(T_L)$
grows with $L$, and iterations increase as $L$
increases (Figure~\ref{fig:pcg_convergence}, bottom left).
\end{remark}

\section{Ground State Computation}
\label{sec:gpe_app}

\subsection{Linear ground state via shifted inverse iteration}
\label{sec:direct_inv}

When $V = V_1$ is fully separable, the tensor-product solver
computes $(-\Delta + V_1)^{-1} b$ directly, which can be used in shifted inverse iteration
(power method for $(-\Delta + V_1 - \sigma)^{-1}$) for computing
the ground state eigenvector of $(-\Delta + V_1)u = \lambda u$.
We use the  potential in~\cite{ChenLuLuZhang2024}:
\begin{equation}\label{eq:separable_V1}
   V_1(x,y,z) = \sum_{k=1}^{3}(x_k^2 + 100\sin^2(\pi x_k/4)).
\end{equation} 
\begin{table}[ht]\centering
\caption{Shifted inverse iteration for the ground state of
$(-\Delta + V_1)\,u = \lambda\,u$ on $[-8,8]^3$,
$Q^{10}$ SEM, FP64.
Potential~\eqref{eq:separable_V1}.
Shift $\sigma = 0.9\,\lambda_{\min}(-\Delta + V_1)$.
Convergence criterion $|\Delta\lambda|/|\lambda| < 10^{-12}$.
Ground state eigenvalue $\lambda_1 = 23.2878438176$.}
\label{tab:inverse_iter}
{\small
\begin{tabular}{@{}llrrrrr@{}}
\toprule
GPU & DoFs & Setup (s) & JIT (s)$^\dagger$ & Iters / time (s) & Peak mem (GB) \\
\midrule
A100  & $599^3$  & 6.2 & 1.0 & 9\,/\,1.25 & 11.2 \\
A100  & $999^3$  & 6.4 & 4.3 & 9\,/\,7.61 & 52.1 \\
GH200 & $999^3$  & 9.3 & 0.9 & 9\,/\,3.50 & 52.1 \\
GH200 & $1099^3$ & 9.2 & 1.2 & 9\,/\,5.03 & 69.4 \\
\bottomrule
\end{tabular}
}

\medskip\footnotesize
$\dagger$\, JAX/XLA compilation time.
At $n{=}999$: GH200 is $2.2{\times}$ faster per solve than A100
(0.39\,s vs.\ 0.85\,s).
$1.33\times10^9$ DoFs converge in 5\,s of iteration time on a single GH200.
\end{table}
 
Table~\ref{tab:inverse_iter} shows that the shifted inverse iteration
converges in 9 iterations regardless of grid size, producing the
ground state of $(-\Delta + V_1)u = \lambda u$.
This eigenfunction will serve as a natural initial guess for the
nonlinear GPE gradient flow below.

\subsection{GPE ground state via Sobolev gradient flow}

We consider the ground state of the defocusing ($\beta > 0$)
Gross--Pitaevskii eigenvalue problem
\begin{equation}\label{eq:gpe}
  -\Delta u + V(\boldsymbol{x})\,u + \beta\,|u|^2\,u = \lambda\,u,
  \quad \boldsymbol{x} \in \Omega \subset \R^3,
  \quad \|u\|_{L^2} = 1,
\end{equation}
with  potential $V \ge 0$ and the parameter $\beta > 0$.
The ground state minimizes  
\begin{equation}\label{eq:gpe_energy}
E(u) = \frac{1}{2}\int_\Omega \bigl(|\nabla u|^2 + V|u|^2\bigr)\,d\boldsymbol{x}
+ \frac{\beta}{4}\int_\Omega |u|^4\,d\boldsymbol{x}.
\end{equation}
For $\beta > 0$, $E$ has a unique positive ground
state~\cite{henning2020sobolev,ChenLuLuZhang2024}, thus 
all computations in this section use real arithmetic.
The ground state can be computed by the normalized gradient flow
(GFDN)~\cite{bao2004computing} or by projected Sobolev gradient
flows~\cite{henning2020sobolev,danaila2017,ChenLuLuZhang2024}.

\subsection{Riemannian gradient flows}
The Sobolev gradient
flows can also be perceived as Riemannian gradient descent methods.
We consider two Riemannian gradient flows on the
$L^2$ unit sphere $\mathcal{S} = \{u : \|u\|_{L^2} = 1\}$.
Both take the form
\begin{equation}\label{eq:riemannian_update}
  u^{n+1} = \frac{u^n - \tau\,\nabla_g^{\mathcal{R}} E(u^n)}
  {\|u^n - \tau\,\nabla_g^{\mathcal{R}} E(u^n)\|_{L^2}},
\end{equation}
where $\nabla_g^{\mathcal{R}} E$ is the Riemannian gradient
with respect to the metric~$g$, and $\tau > 0$ is the step size.
They differ only in the choice of~$g$. We consider just two of them.

\medskip\noindent
\textbf{Modified $H^1$ flow}~\cite{ChenLuLuZhang2024}.
The metric is
\begin{equation}\label{eq:h1_metric}
  g_{H^1}(w, z) = (\nabla w, \nabla z) + \alpha\,(w, z),
  \qquad \alpha > 0.
\end{equation}
Each iteration requires solving
$(-\Delta + \alpha I)\,w = (-\Delta + V + \beta u^2)\,u$
and a projection step, amounting to two direct solves
of $(-\Delta + \alpha I)^{-1}$ per iteration,
a separable tensor-product operation at cost $O(N^{4/3})$ in 3D.

\medskip\noindent
\textbf{$a_u$ flow}~\cite{henning2020sobolev}.
The metric depends on the current iterate~$u$:
\begin{equation}\label{eq:au_metric}
  g_{a_u}(w, z) = (\nabla w, \nabla z) + (w, V\,z)
  + \beta\,(w, u^2\,z).
\end{equation}
Each iteration requires solving
$(-\Delta + V + \beta u^2)^{-1}$.
We solve this by PCG with preconditioner $(-\Delta + V)^{-1}$,
which is a separable tensor-product solve. We refer to~\cite{henning2020sobolev,ChenLuLuZhang2024} for more details about $a_u$ flow.
The $a_u$ metric adapts to the current iterate. It reduces the
iteration count, but each step is more expensive because of the PCG
solve.
The tensor-product solver enters both methods:
(i) shifted inverse iteration provides the eigenfunction of the
linearized problem $(-\Delta + V_1)u = \lambda u$ as an initial guess
(Section~\ref{sec:direct_inv}),
(ii) the modified $H^1$ flow requires $(-\Delta + \alpha I)^{-1}$
(direct solve), and
(iii) the $a_u$ flow uses $(-\Delta + V)^{-1}$ as a PCG preconditioner.

\subsection{Numerical results}

When FP32 is used below for the PCG solves in the $a_u$ flow, the
offline eigendecompositions per axis are still
computed in FP64. We consider the ground state of  \eqref{eq:gpe} for a potential $V=V_1$ with $V_1$ given in \eqref{eq:separable_V1}. 

\begin{example}[GPE ground state with eigenfunction initial guess]
\label{ex:gpe}
Since the GPE ground state approaches the linear ground state
as $\beta \to 0$, the eigenfunction from
Table~\ref{tab:inverse_iter} provides a good initial guess for
moderate $\beta$.
Table~\ref{tab:h1_flow} compares the modified $H^1$
flow~\cite{ChenLuLuZhang2024} with two initial guesses:
a constant function $u_0 \equiv c$ (normalized s.t. $\|u_0\|_{L^2}=1$) and the linear ground
state from shifted inverse iteration.
Both use $Q^{20}$ SEM with $599^3$ DoFs,
$\alpha = 20$, $\tau = 0.1$, tested for
$\beta = 10$ and $100$ on an NVIDIA GH200 96\,GB GPU.
For moderate $\beta$, the nonlinear ground state remains close to the
linear eigenfunction, so the eigenfunction initial guess starts the
flow near the solution. Figure~\ref{fig:gpe_states} illustrates the progression: at $\beta = 100$
the lattice structure is already emerging but the solution remains
concentrated near the center, while at $\beta = 1600$ the ground state
has spread into a broad  profile far from the linear eigenfunction.
As shown in Table~\ref{tab:h1_flow}, for $\beta=10, 100$, modified $H^1$ flow with both initializations converges to the same
energy and eigenvalue, and the eigenfunction initial guess
saves 35--71\% of total time with 39--77\% fewer iterations.
Figure~\ref{fig:h1_convergence} illustrates that both $H^1$ flow and $a_u$ flow for $\beta = 100$ can benefit from
the eigenfunction initial guess. But for $\beta=1600$, as shown in  Table~\ref{tab:au_vs_h1} ($N=599^3$) and Figure~\ref{fig:au_vs_h1} ($N=999^3$), the eigenfunction is no longer a good initial guess. Both Figure~\ref{fig:h1_convergence} and Figure~\ref{fig:au_vs_h1} also show that $a_u$ flow can be efficiently implemented by the PCG with the proposed preconditioner $(-\Delta+V_1)^{-1}$. 
\end{example}

\begin{figure}[ht]\centering
\includegraphics[width=0.32\textwidth]{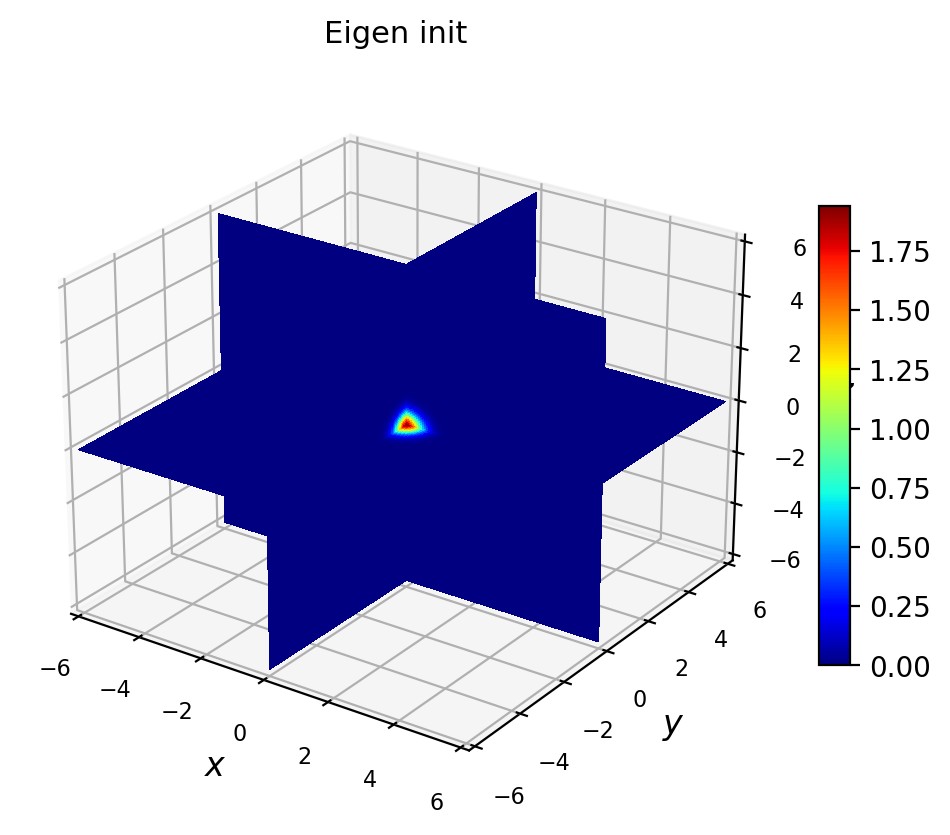}\hfill
\includegraphics[width=0.32\textwidth]{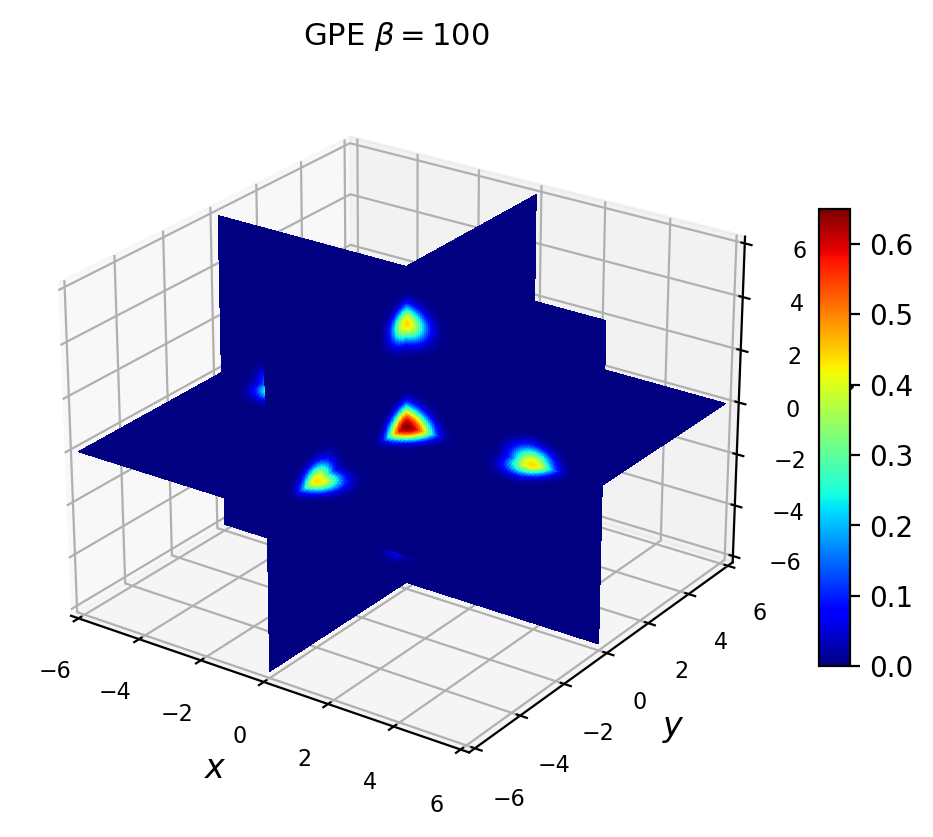}\hfill
\includegraphics[width=0.32\textwidth]{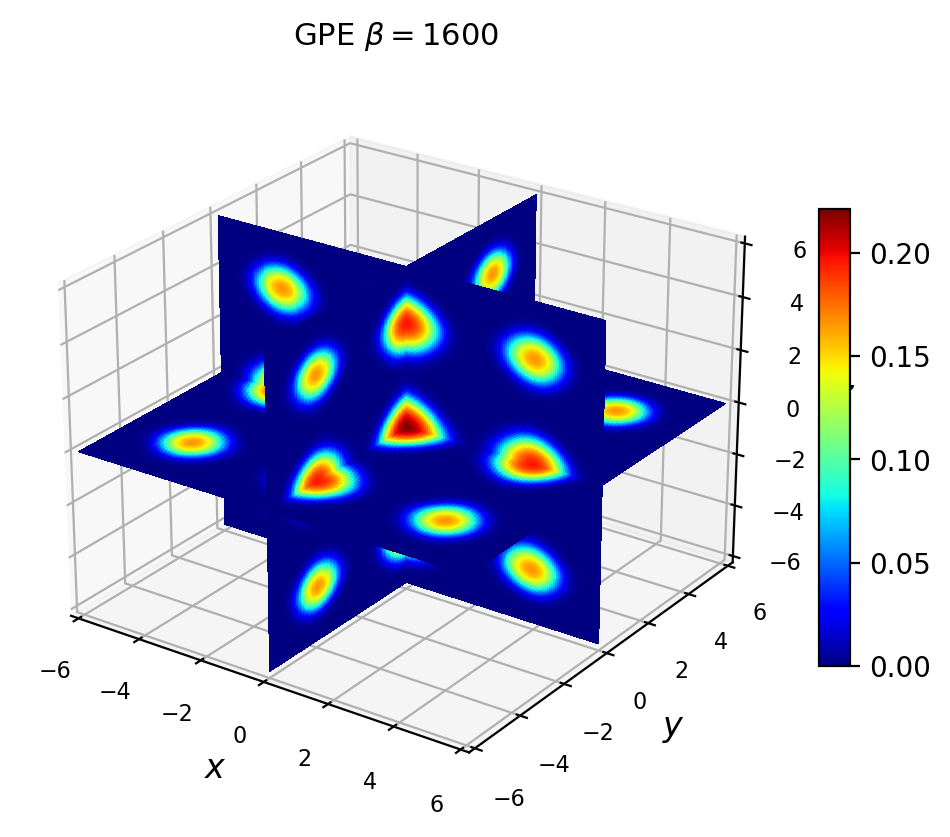}
\caption{Cross-sections of $u$ at $x{=}0$, $y{=}0$, $z{=}0$,
plotted on $[-6,6]^3$.
$Q^{20}$ SEM, $199^3$ DoFs, computed on $[-8,8]^3$.
Each panel uses its own color scale.
\emph{Left:} eigenfunction initial guess (linear ground state, $\beta{=}0$).
\emph{Center:} GPE ground state, $\beta{=}100$.
\emph{Right:} GPE ground state, $\beta{=}1600$.}
\label{fig:gpe_states}
\end{figure}

\begin{table}[ht]\centering
\caption{Modified $H^1$ flow for the GPE ground state on $[-8,8]^3$,
$Q^{20}$ SEM, FP64, $\alpha = 20$, step size $\tau = 0.1$, $599^3$ DoFs,
GH200. Potential~\eqref{eq:separable_V1}.
Convergence criterion:
$|E_k - E^*|/|E^*| < 10^{-14}$. 
The total computation time includes both off-line and iteration time.}
\label{tab:h1_flow}
{\small
\begin{tabular}{@{}rlrrrl@{}}
\toprule
$\beta$ & Init & $H^1$ iters & Total (s) & Energy & Eigenvalue \\
\midrule
10 & constant      & 745 & 122.0 & 14.1965761916 & 32.4916917439 \\
   & eigenfunction & 175 &  34.8 & 14.1965761916 & 32.4916917439 \\
\midrule
100 & constant      & 660 & 111.4 & 20.6824463703 & 47.7831207152 \\
    & eigenfunction & 405 &  71.9 & 20.6824463703 & 47.7831207152 \\
\bottomrule
\end{tabular}
}
\end{table}

\begin{figure}[ht]\centering
\includegraphics[width=\textwidth]{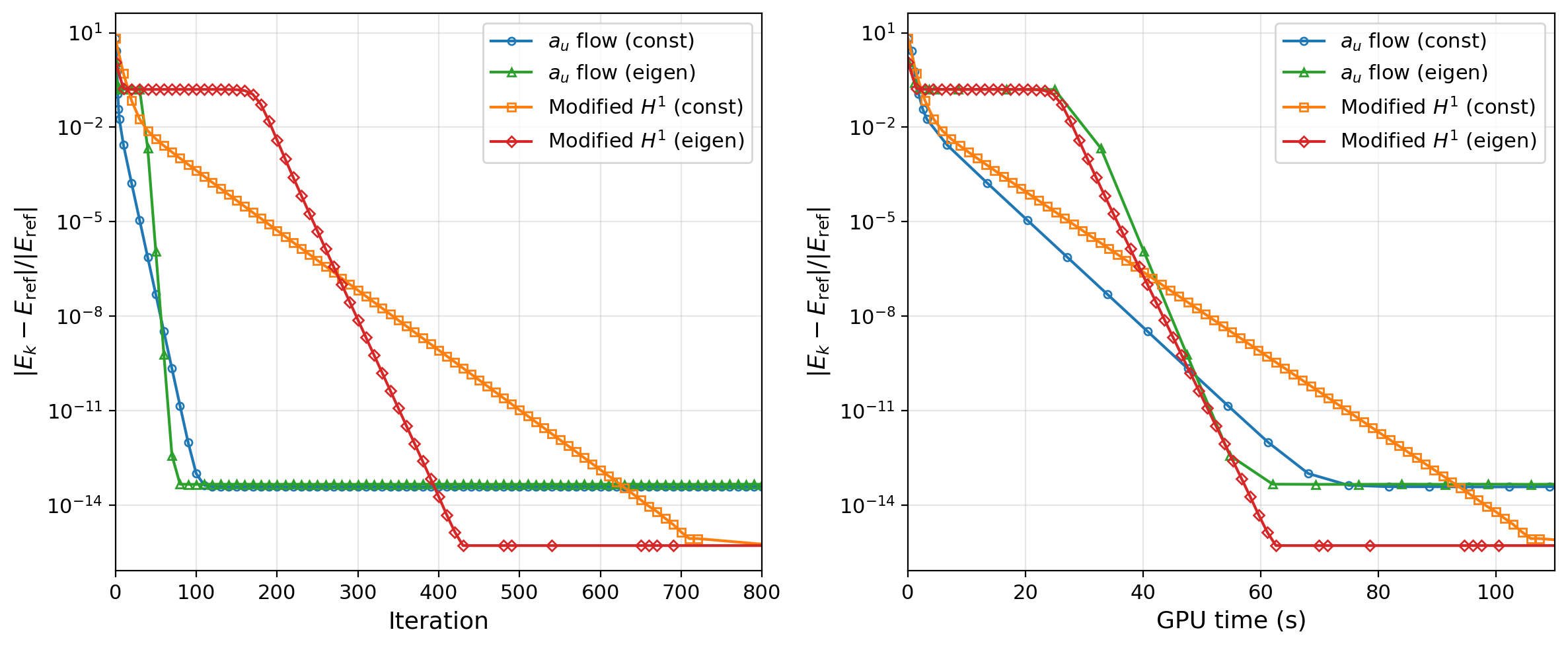}
\caption{Comparison of $a_u$ flow and modified $H^1$ flow
for $\beta = 100$,
$Q^{20}$ SEM, $599^3 \approx 2.1\times10^8$ DoFs, GH200.
Left: relative energy error vs.\ iteration.
Right: relative energy error vs.\ GPU time.}
\label{fig:h1_convergence}
\end{figure}

\begin{figure}[ht]\centering
\includegraphics[width=\textwidth]{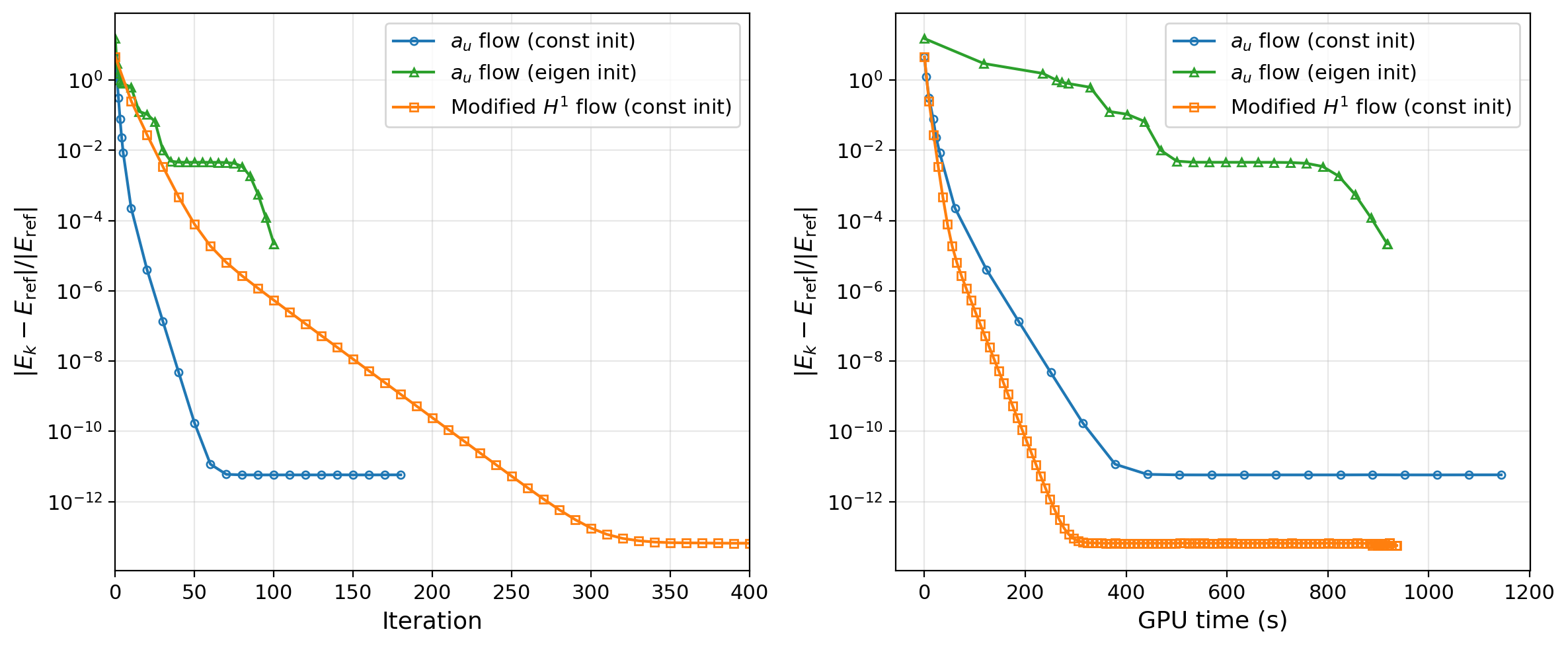}
\caption{Comparison of $a_u$ flow and modified $H^1$ flow for
$\beta = 1600$, $Q^{20}$ SEM, $999^3 \approx 10^9$ DoFs, constant
initial guess, GH200.
Left: relative energy error vs.\ iteration.
Right: relative energy error vs.\ GPU time.
The reference energy $E^* = 33.80227900547$
is from~\cite{ChenLuLuZhang2024}.}
\label{fig:au_vs_h1}
\end{figure}

\begin{table}[ht]\centering
\caption{Comparison of $a_u$ flow and modified $H^1$ flow for
$\beta = 1600$, $Q^{20}$ SEM, $599^3 \approx 2.1\times10^8$ DoFs,
FP64, $\alpha{=}20$, step size $\tau{=}0.1$ for modified $H^1$ and $\tau{=}1$ for $a_u$,
GH200.}
\label{tab:au_vs_h1}
{\small
\begin{tabular}{@{}llrrrr@{}}
\toprule
Method & Initilization & Iters  & Linear solves & Time (s) & Energy \\
\midrule
$a_u$ flow & constant      &  66 &  790 &  72.8 & 33.80227900550 \\
$a_u$ flow & eigenfunction & 147 & 1846 & 171.9 & 33.80227900550 \\
Modified $H^1$ & constant  & 272 &  544 &  42.6 & 33.80227900551 \\
Modified $H^1$ & eigenfunction & \multicolumn{4}{l}{not converging} \\
\bottomrule
\end{tabular}
}

\medskip\footnotesize
Convergence criterion: $|E_k - E^*|/|E^*| < 10^{-12}$,
where $E^* = 33.80227900547423$.
\end{table}

\section{Hamiltonian Simulation in 3D}
\label{sec:hamiltonian}

We consider the Schr\"odinger equation
\begin{equation}\label{eq:tdse}
  i\,\partial_t \psi = H\,\psi, \quad
  H = -\Delta + V, \quad
  \psi(\cdot, 0) = \psi_0,
\end{equation}
on $\Omega \subset \R^3$ with either periodic or Dirichlet boundary
conditions.
When $V$ is separable, the tensor-product solver gives the exact
propagator $e^{-iH\Delta t}$. When $V$ is non-separable, it
provides the propagator for each operator-splitting sub-step.

\subsection{Splitting methods and classical implementation}

We split $H = A + B$ with $A = -\Delta$ (kinetic) and $B = V$
(potential).
The standard Strang splitting
$e^{-iA\Delta t/2}\,e^{-iB\Delta t}\,e^{-iA\Delta t/2}$
is second-order in $\Delta t$, but its operator-norm error
grows with the spectral radius of $[A,B]$, which may scale with $N$.
Two recent quantum algorithms consider the interaction-picture Hamiltonian
$H_I(s) = e^{iAs}\,B\,e^{-iAs}$ and apply the Magnus expansion
to the time-ordered evolution:
\begin{itemize}[nosep]
\item The \emph{qHOP} algorithm~\cite{AnFangLin2022}
  uses the first-order Magnus truncation, achieving
  $O(\Delta t^2)$ with an error constant independent of $N$.
\item The \emph{Magnus-2} algorithm~\cite{FangLiuSarkar2025}
  adds a commutator correction (second-order Magnus),
  achieving $O(\Delta t^4)$
  superconvergence~\cite{BornsWeilFangZhang2026,FangZhang2025preprint}.
\end{itemize}
Both methods require a matrix exponential of a sum of
interaction-picture Hamiltonians, which on a quantum computer
is implemented via LCU~\cite{AnFangLin2022}.
Classically, $e^{iAs}$ is computed by the tensor-product
propagator and $e^{-iB\Delta t}$ is pointwise, so each factor
$\exp(-i\,w_k\Delta t\,H_I(s_k))$ can be evaluated exactly:
\[
  e^{-i\,w_k\Delta t\,H_I(s_k)}
  = e^{iAs_k}\,e^{-i\,w_k\Delta t\,B}\,e^{-iAs_k},
\]
which leads to two families of methods, implemented
via tensor-product propagations.

\medskip\noindent
\textbf{$O(\Delta t^2)$ approximated qHOP via product formula.}
The qHOP approximation~\cite{AnFangLin2022}
replaces the time-ordered exponential by
the first-order Magnus truncation in the interaction picture
(see~\cite[Eq.~(2)]{FangLiuSarkar2025}):
\[
  e^{-iH\Delta t} = e^{-iA\Delta t}\,
  \mathcal{T}\exp\!\Bigl({-i\!\int_0^{\Delta t}\! H_I(s)\,ds}\Bigr)
  \approx e^{-iA\Delta t}\,
  \exp\!\Bigl({-i\,\Delta t \sum_{k=1}^{M} w_k\,H_I(s_k)}\Bigr),
\]
where $s_k \in [0, \Delta t]$ are Gauss--Legendre nodes and   $w_k$ ($\sum_{k=1}^M w_k = 1$) are the corresponding normalized Gauss--Legendre weights.
We implement this classically via the product approximation
\begin{equation}\label{eq:qhop_product}
  \exp\!\Bigl(-i\,\Delta t \sum_{k=1}^{M} w_k\,H_I(s_k)\Bigr)
  \approx \prod_{k=1}^{M} e^{-i\,w_k\Delta t\,H_I(s_k)}.
\end{equation}
Since $H_I(s_k) = e^{iAs_k}B\,e^{-iAs_k}$,
each factor in the product is computed \emph{exactly} as
$e^{-i\,w_k\Delta t\,H_I(s_k)}
= e^{i A s_k}\,e^{-i\,w_k\Delta t\,B}\,e^{-i A s_k}$
(two tensor-product propagations plus one pointwise
multiplication).
Combining~\eqref{eq:qhop_product} with
$e^{-iH\Delta t} = e^{-iA\Delta t}\,
\mathcal{T}\exp(\cdots)$
gives the full time step
$e^{-iH\Delta t}
\approx e^{-iA\Delta t}\,
\prod_{k=1}^{M} e^{-i\,w_k\Delta t\,H_I(s_k)}$,
where the approximation comes from the quadrature and
product formula~\eqref{eq:qhop_product}.

\begin{remark}[Order of the product formula]
\label{rem:product_order}
The product formula~\eqref{eq:qhop_product} replaces
$\exp(\text{sum})$ by a product of exponentials,
which is in general a first-order Lie--Trotter
approximation with error governed by the commutators
$[A_j, A_k]$ where $A_k = -iw_k\Delta t\,H_I(s_k)$.
In Lie--Trotter splitting of two \emph{different}
operators $A$ and $B$, the commutator $[A,B]$ is a fixed  operator, giving first-order global error $O(\Delta t)$.
Here, however, all the operators $H_I(s_k)$ have
$s_k \in [0,\Delta t]$, so as $\Delta t \to 0$,
$H_I(s_k) \to H_I(0) = B$ for all $k$, and the commutators
$[H_I(s_j), H_I(s_k)] \to [B, B] = 0$.
More precisely, smoothness of $H_I(s)$ in $s$ gives
$[H_I(s_j), H_I(s_k)] = O(|s_j - s_k|) = O(\Delta t)$,
so each $[A_j, A_k]
= w_j w_k \Delta t^2\,[H_I(s_j), H_I(s_k)] = O(\Delta t^3)$.
The product formula error is therefore $O(\Delta t^3)$
per step and $O(\Delta t^2)$ globally,
matching the qHOP approximation order.
\end{remark}

\medskip\noindent
\textbf{$O(\Delta t^4)$ approximated Magnus-2 via Yoshida composition.}
The Magnus-2 algorithm~\cite{FangLiuSarkar2025} replaces the
first-order Magnus truncation with the second-order one:
\[
  e^{-iH\Delta t}
  \approx e^{-iA\Delta t}\,
  \exp\!\bigl(\Omega_1 + \Omega_2\bigr),
\]
where $\Omega_1 = -i\,\Delta t\sum_{k=1}^M w_k\,H_I(s_k)$
is the qHOP exponent and
\[
  \Omega_2 = -\tfrac{1}{2}\int_0^{\Delta t}\!\int_0^{s_1}
  [H_I(s_1),H_I(s_2)]\,ds_2\,ds_1
\]
is the commutator correction, achieving $O(\Delta t^4)$
superconvergence.
On a quantum computer, $e^{\Omega_1+\Omega_2}$ is implemented
via LCU.
Classically, it is impractical to compute $e^{\Omega_1+\Omega_2}$ directly.
Instead, we apply the fourth-order Yoshida
splitting~\cite{yoshida1990} to the approximated qHOP:
\begin{equation}\label{eq:yoshida}
  e^{-iH\Delta t}
  \approx
  P(\gamma_1\Delta t)\,
  P(\gamma_2\Delta t)\,
  P(\gamma_1\Delta t),
  \qquad
  \gamma_1 = \frac{1}{2 - 2^{1/3}}, \quad
  \gamma_2 = \frac{-2^{1/3}}{2 - 2^{1/3}},
\end{equation}
where $P(h)$ denotes one approximated qHOP step with
step size $h$, i.e.,
$P(h) = e^{-iAh}\prod_{k=1}^{M}
e^{-i\,w_k h\,H_I(s_k)}$,
and $\gamma_1 + \gamma_2 + \gamma_1 = 1$.
Yoshida cancels the $O(h^3)$ leading error
($2\gamma_1^3 + \gamma_2^3 = 0$),
yielding $O(\Delta t^4)$ global accuracy,
the same order as the original Magnus-2 method.
Although~\eqref{eq:yoshida} is derived by applying Yoshida
composition to the approximated qHOP rather than by directly
approximating $e^{\Omega_1+\Omega_2}$, it achieves the same
$O(\Delta t^4)$ rate as the original Magnus-2, and can therefore
be regarded as an approximation to  Magnus-2.
\begin{remark}[Classical splitting as limiting cases]
qHOP with $M{=}1$ (midpoint rule) recovers the Strang splitting~\cite{AnFangLin2022}.
At the next order, our approximated Magnus-2 with $M{=}1$ reduces to
the plain Yoshida method.
Thus Strang and Yoshida can be perceived as the $M{=}1$ special
cases of such approximated versions of qHOP and Magnus-2,
respectively, while larger $M$ reduces the error constant.
\end{remark}

\begin{remark}[Cost per step]                                                      
  \label{rem:cost_naive_merged}  
  For $M \ge 2$, the  approximated qHOP in~\eqref{eq:qhop_product} has a direct count of $(2M+1)$ tensor-product propagations per step, i.e., number of    multiplying $e^{\pm iAs_k}$ per step.  For $M = 1$, the optimized Strang form          $e^{-iV\Delta t/2}\,e^{-iA\Delta t}\,e^{-iV\Delta t/2}$  needs only $1$ tensor-product propagation per step. A better implementation for $M>1$ is to merge two adjacent factors, e.g., for $M=3$, the merged form of the qHOP is 
  \[ e^{-iA(\Delta t - s_3)} e^{-iw_3\Delta t\,B} e^{-iA(s_3 - s_2)}\,                 
    e^{-iw_2\Delta t\,B}e^{-iA(s_2 - s_1)} e^{-iw_1\Delta t\,B} e^{-iAs_1},\]                          which has only $(M+1)$ A-propagations per step.                      By merging $e^{-iAs_1}$ at  step $n$  with  $e^{-iA(\Delta t - s_M)}$ at   step $n+1$, the cost would be                                       $M$ A-propagations per time step.
 Applying the same merging strategy to each of the three Yoshida sub-steps and across consecutive time steps, the approximated Magnus-2 costs $3M$ A-propagations per step asymptotically for $M\ge 2$, and $3$ A-propagations plus $3$ pointwise $B$-multiplications per step for $M=1$ (plain Yoshida).   \end{remark}

\subsection{Exact propagator as reference for testing splitting methods}
\label{sec:ham_separable}

When $V = V_1$ is  separable, the eigendecomposition
$H = T\,\Lambda\,T^{-1}$ gives the exact propagator
\begin{equation}\label{eq:exact_prop}
  e^{-iH\Delta t}\,\psi = T\,(e^{-i\Lambda\Delta t} \odot (T^{-1}\,\psi)),
\end{equation}
at cost $O(N^{4/3})$ in 3D.
We test with the separable potential~\eqref{eq:separable_V1}
from  Section~\ref{sec:direct_inv} on $[-8,8]^3$,
Dirichlet BC, $Q^k$ SEM, and the exact
propagator~\eqref{eq:exact_prop} as reference.

Most existing numerical tests of these methods are in 1D with
$N \le 1024$, using dense matrix exponentials (\texttt{expm}).
Our tensor-product solver provides the exact propagator
as a reference solution at ${\sim}10^8$ DoFs, enabling
large-scale 3D validation of the splitting convergence rates.
In the rest of this section, ``qHOP'' and ``Magnus-2'' refer to the
approximated versions implemented via the product
formula~\eqref{eq:qhop_product} and Yoshida
composition~\eqref{eq:yoshida}, respectively.

\begin{example}[Order of approximated qHOP]
\label{ex:splitting}
Table~\ref{tab:splitting} reports qHOP with $M = 1$ (Strang), $3$, $5$, and $7$
on a GH200 GPU. The discretization uses $Q^{10}$ SEM with $499^3$
DoFs, complex128 arithmetic, and $T = 0.1$.
The initial condition is $\psi_0(\boldsymbol{x})
= \prod_{k=1}^{3}\sin(\pi(x_k{+}L)/2L)$.
The observed convergence is $O(\Delta t^2)$ for all $M$.
Relative to Strang, the error constant is smaller by
${\sim}5$ ($M{=}3$), ${\sim}13$ ($M{=}5$), and ${\sim}24$ ($M{=}7$). See Remark \ref{rem:cost_naive_merged} 
for cost comparison.
\end{example}

\begin{table}[ht]\centering
\caption{qHOP splitting error $\|\psi_{\rm split}(T){-}\psi_{\rm exact}(T)\|_2$
on GH200, $Q^{10}$ SEM, $499^3$ DoFs, complex128, $T{=}0.1$.
$M{=}1$ is the Strang splitting. $t$\,(s) is the GPU time.}
\label{tab:splitting}
{\small
\begin{tabular*}{\textwidth}{@{\extracolsep{\fill}}rrrrrrrrrrrrrr@{}}
\toprule
& \multicolumn{3}{c}{$M=1$ (Strang)} & \multicolumn{3}{c}{$M=3$} & \multicolumn{3}{c}{$M=5$} & \multicolumn{3}{c}{$M=7$} \\
\cmidrule(lr){2-4}\cmidrule(lr){5-7}\cmidrule(lr){8-10}\cmidrule(lr){11-13}
$\Delta t$ & Error & Rate & $t$ (s) & Error & Rate & $t$ (s) & Error & Rate & $t$ (s) & Error & Rate & $t$ (s) \\
\midrule
0.1   & 5.85e-1 & --   & 0.1 & 1.07e-1 & --   & 0.4 & 4.25e-2 & --   & 0.5 & 2.28e-2 & --   & 0.7 \\
0.01  & 5.22e-3 & 2.05 & 0.9 & 1.02e-3 & 2.02 & 3.0 & 4.11e-4 & 2.01 & 4.8 & 2.21e-4 & 2.01 & 6.7 \\
0.005 & 1.30e-3 & 2.01 & 1.9 & 2.55e-4 & 2.00 & 5.8 & 1.03e-4 & 2.00 & 9.6 & 5.52e-5 & 2.00 & 13  \\
0.001 & 5.21e-5 & 2.00 & 9.5 & 1.02e-5 & 2.00 & 29  & 4.11e-6 & 2.00 & 48  & 2.21e-6 & 2.00 & 67  \\
\bottomrule
\end{tabular*}
}
\end{table}

\subsection{Non-separable potential: manufactured exact solution}
\label{sec:ham_nonseparable}

When $V = V_1 + V_2$ is not fully separable,  we use the PCG solver from Section~\ref{sec:pcg_inv} to
compute the ground state $(\lambda_1, u_1)$ of
$H = -\Delta + V_1 + V_2$.  The function
\begin{equation}\label{eq:manufactured}
  \psi_{\rm exact}(\boldsymbol{x}, t) = e^{-i\lambda_1 t}\,u_1(\boldsymbol{x})
\end{equation}
is a stationary solution of~\eqref{eq:tdse}, against which
we verify splitting convergence orders. We test with the stirrer potential~\eqref{eq:stirrer}
on $[-8,8]^3$, Dirichlet BC, $Q^{20}$ SEM, and we also consider splitting
$H = A + B$ with $A = -\Delta + V_1$ 
and $B = V_2$.

\subsubsection{Multi-level generation of the reference solution}

The ground state $(\lambda_1, u_1)$ of $H = A + B$ is computed by
shifted inverse iteration: each step solves
$(H - \sigma I)\,w = u$ via PCG with the tensor-product
preconditioner $A^{-1} = (-\Delta + V_1)^{-1}$
from Section~\ref{sec:pcg_inv}.
The shift $\sigma = 0.9\,\lambda_{\min}(A)$ keeps $H - \sigma I$
positive definite (since $V_2 \ge 0$) and close enough to $\lambda_1$
for rapid convergence.
The initial guess is the ground state eigenfunction of $A$, obtained
from the tensor-product eigendecomposition.

To efficiently reach large grid sizes, we use a multi-level
strategy: solve on a coarse grid, interpolate the eigenvector to the
next finer grid via tensor-product linear interpolation, and continue the inverse iteration.
Each refinement inherits a good initial guess, reducing both the number
of inverse iterations and the PCG iterations per step.
Table~\ref{tab:multilevel} demonstrates this for the stirrer
potential~\eqref{eq:stirrer} on an A100 GPU.
At the finest level ($499^3$ DoFs), only 11 inverse iterations with
21~PCG iterations each are needed, a total of 231 applications
of $(-\Delta + V_1)^{-1}$.
The entire multi-level computation takes about one minute, producing
a ground state accurate to $|\Delta\lambda|/|\lambda| < 10^{-13}$.

\begin{table}[ht]\centering
\caption{Multi-level shifted inverse iteration for the ground state of
$H = -\Delta + V_1 + V_2$,
stirrer potential~\eqref{eq:stirrer}, $Q^{20}$ SEM, FP64, A100.
Preconditioner: $(-\Delta + V_1)^{-1}$, PCG tolerance $10^{-12}$,
shift $\sigma = 0.9\,\lambda_{\min}(A)$.
Ground state eigenvalue $\lambda_1 = 5.286155366963$.}
\label{tab:multilevel}
{\small
\begin{tabular}{@{}rrrrrr@{}}
\toprule
$n$ (DoFs) & Setup (s) & Interpolation (s) & Inverse iters & PCG/iter & Total $A^{-1}$ \\
\midrule
$99^3$  & 6.1 & ---  & 30 & 75 & 2242 \\
$199^3$ & 2.6 & 0.3  & 21 & 24 & 504  \\
$499^3$ & 2.6 & 0.5  & 11 & 21 & 231  \\
\bottomrule
\end{tabular}
}

\medskip\footnotesize
Initial guess at $99^3$: ground state of $A = -\Delta + V_1$.
{\bf Total wall time: ${\sim}62$\,s on A100}.
\end{table}
 
\subsubsection{Testing splitting methods for Hamiltonian simulations}

\begin{example}[Order of approximated qHOP for non-separable stirrer potential]
\label{ex:splitting_nonsep}
We test qHOP with $M = 1$ (Strang), $3$, $5$, and $7$,
$499^3$ DoFs, complex128, $T = 0.1$, on a GH200 GPU.
The initial condition is $\psi_0 = u_1$ (ground state eigenvector,
$\lambda_1 = 5.286155366963$).
Table~\ref{tab:splitting_nonsep} compares two splittings:
$A = -\Delta$, $B = V_1 + V_2$ (top half), and
$A = -\Delta + V_1$, $B = V_2$ (bottom half),
where $A$ includes the separable part of the potential.
Both show $O(\Delta t^2)$ convergence for all $M$, but the
splitting $A = -\Delta + V_1$ gives ${\sim}1.5$--$3{\times}$
smaller errors because $\|V_2\| \ll \|V_1 + V_2\|$.
For qHOP as a quantum algorithm, $A = -\Delta$ is required
since $A$ must be the operator that can be fast-forwarded
on a quantum computer.
For classical ODE/PDE solvers, including the separable potential
in $A$ is advantageous.
\end{example}

\begin{table}[ht]\centering
\caption{qHOP splitting error for the non-separable
stirrer potential~\eqref{eq:stirrer},
$Q^{20}$ SEM, $499^3$ DoFs,
complex128, $T{=}0.1$, GH200. $t$\,(s) is the GPU time. 
Reference: manufactured solution~\eqref{eq:manufactured}.}
\label{tab:splitting_nonsep}
{\small
\begin{tabular*}{\textwidth}{@{\extracolsep{\fill}}rrrrrrrrrrrrrr@{}}
\toprule
& \multicolumn{3}{c}{$M=1$ (Strang)} & \multicolumn{3}{c}{$M=3$} & \multicolumn{3}{c}{$M=5$} & \multicolumn{3}{c}{$M=7$} \\
\cmidrule(lr){2-4}\cmidrule(lr){5-7}\cmidrule(lr){8-10}\cmidrule(lr){11-13}
$\Delta t$ & Error & Rate & $t$ (s) & Error & Rate & $t$ (s) & Error & Rate & $t$ (s) & Error & Rate & $t$ (s) \\
\midrule
\multicolumn{13}{c}{\textit{Splitting $A = -\Delta$, $B = V_1 + V_2$}} \\
\midrule
0.1   & 8.04e-3 & --   & 0.7 & 3.93e-4 & --   & 1.0 & 1.58e-4 & --   & 1.1 & 8.49e-5 & --   & 1.4 \\
0.01  & 7.73e-5 & 2.02 & 1.6 & 3.90e-6 & 2.00 & 3.5 & 1.57e-6 & 2.00 & 5.6 & 8.44e-7 & 2.00 & 7.8 \\
0.005 & 1.93e-5 & 2.00 & 2.6 & 9.74e-7 & 2.00 & 6.5 & 3.93e-7 & 2.00 & 11  & 2.11e-7 & 2.00 & 15  \\
0.001 & 7.73e-7 & 2.00 & 11  & 3.94e-8 & 1.99 & 30  & 1.64e-8 & 1.97 & 51  & 9.46e-9 & 1.93 & 72  \\
\midrule
\multicolumn{13}{c}{\textit{Splitting $A = -\Delta + V_1$, $B = V_2$}} \\
\midrule
0.1   & 4.75e-3 & --   & 0.9 & 1.37e-4 & --   & 1.0 & 5.50e-5 & --   & 1.3 & 2.96e-5 & --   & 1.5 \\
0.01  & 4.52e-5 & 2.02 & 1.7 & 1.35e-6 & 2.01 & 3.7 & 5.47e-7 & 2.00 & 5.6 & 2.95e-7 & 2.00 & 7.7 \\
0.005 & 1.13e-5 & 2.00 & 2.6 & 3.40e-7 & 1.99 & 6.6 & 1.38e-7 & 1.99 & 10  & 7.50e-8 & 1.98 & 15  \\
0.001 & 4.52e-7 & 2.00 & 10  & 1.61e-8 & 1.90 & 30  & 8.99e-9 & 1.70 & 49  & 7.19e-9 & 1.46 & 71  \\
\bottomrule
\end{tabular*}
}
\end{table}

\begin{example}[Order of approximated Magnus-2 for non-separable stirrer potential]
\label{ex:magnus2_nonsep}
We use the same setup, but with $T = 1$.
Table~\ref{tab:magnus2_nonsep} compares the two splittings
as in Table~\ref{tab:splitting_nonsep}.
Both show $O(\Delta t^4)$ convergence for $M \ge 3$.
The splitting $A = -\Delta + V_1$, $B = V_2$ again gives smaller
errors, with the advantage growing with $M$.
\end{example}

\begin{table}[ht]\centering
\caption{Magnus-2 via Yoshida splitting error for the
non-separable stirrer potential~\eqref{eq:stirrer}, $Q^{20}$ SEM,
$499^3$ DoFs, complex128, $T{=}1$, GH200. $t$\,(s) is the GPU time.
Top: $A = -\Delta$, $B = V_1 + V_2$.
Bottom: $A = -\Delta + V_1$, $B = V_2$, which gives
${\sim}1.5$--$3{\times}$ smaller errors.
Reference: manufactured solution via multi-level inverse iteration.}
\label{tab:magnus2_nonsep}
{\small
\begin{tabular*}{\textwidth}{@{\extracolsep{\fill}}rrrrrrrrrrrrrr@{}}
\toprule
& \multicolumn{3}{c}{$M=1$ (Yoshida)} & \multicolumn{3}{c}{$M=3$} & \multicolumn{3}{c}{$M=5$} & \multicolumn{3}{c}{$M=7$} \\
\cmidrule(lr){2-4}\cmidrule(lr){5-7}\cmidrule(lr){8-10}\cmidrule(lr){11-13}
$\Delta t$ & Error & Rate & $t$ (s) & Error & Rate & $t$ (s) & Error & Rate & $t$ (s) & Error & Rate & $t$ (s) \\
\midrule
\multicolumn{13}{c}{\textit{Splitting $A = -\Delta$, $B = V_1 + V_2$}} \\
\midrule
0.2   & 8.63e-2 & --   & 2.2 & 5.36e-3 & --   & 5.1 & 4.61e-4 & --   & 8.2 & 1.02e-4 & --   & 11 \\
0.125 & 1.30e-2 & 4.03 & 3.1 & 5.44e-4 & 4.87 & 7.7 & 2.93e-5 & 5.86 & 13  & 1.32e-5 & 4.35 & 18 \\
0.1   & 5.34e-3 & 3.99 & 3.7 & 1.33e-4 & 6.31 & 9.5 & 1.10e-5 & 4.39 & 16  & 5.02e-6 & 4.33 & 22 \\
0.05  & 3.39e-4 & 3.98 & 6.9 & 3.25e-6 & 5.35 & 19  & 6.77e-7 & 4.02 & 30  & 3.00e-7 & 4.06 & 44 \\
\midrule
\multicolumn{13}{c}{\textit{Splitting $A = -\Delta + V_1$, $B = V_2$}} \\
\midrule
0.2   & 2.48e-2 & --   & 2.2 & 5.06e-3 & --   & 5.1 & 4.21e-4 & --   & 8.1 & 7.04e-5 & --   & 11 \\
0.125 & 3.75e-3 & 4.02 & 3.1 & 5.15e-4 & 4.86 & 7.8 & 1.46e-5 & 7.15 & 13  & 7.29e-6 & 4.82 & 17 \\
0.1   & 1.50e-3 & 4.11 & 3.8 & 1.20e-4 & 6.53 & 9.5 & 4.18e-6 & 5.60 & 16  & 2.37e-6 & 5.04 & 21 \\
0.05  & 1.04e-4 & 3.85 & 6.5 & 8.50e-7 & 7.14 & 18  & 2.37e-7 & 4.14 & 30  & 1.29e-7 & 4.20 & 42 \\
\bottomrule
\end{tabular*}
}
\end{table}

\section{Multi-Body Hamiltonian Simulation (4D, 6D, 9D)}
\label{sec:multibody}

In this section we consider multi-particle systems with pairwise
interactions in 4D, 6D, and 9D.
As in Section~\ref{sec:ham_nonseparable}, the reference solution is
$\psi_{\rm exact} = e^{-i\lambda_1 t}\,u_1$, with the ground state
computed by PCG using $(-\Delta + V_{\mathrm{trap}})^{-1}$ as
preconditioner.

\subsection{Two particles in 2D with Coulomb interaction (4D)}
\label{sec:coulomb4d}

We consider two quantum particles in a 2D harmonic trap with
Coulomb interaction.
The Coulomb singularity $1/|\bx_1 - \bx_2|$
at $\bx_1 = \bx_2$ needs be regularized unless special transformation is used.
One approach is the cell-averaged potential of~\cite{coulomb2026},
which integrates $1/|\bx_1 - \bx_2|$ against the basis functions.
We use the simpler soft Coulomb potential
$c/\sqrt{|\bx_1 - \bx_2|^2 + \delta^2}$ with a small  
parameter $\delta > 0$, which is smooth for any $\delta > 0$, so
we expect $O(\Delta t)$ for Lie--Trotter and $O(\Delta t^2)$
for Strang, although Lie--Trotter is only $O(\Delta t^{1/4})$
for the exact (unregularized) Coulomb
potential~\cite{FangWuSoffer2026}.
The equation is
\begin{equation}\label{eq:coulomb}
  i\,\partial_t \psi =
  \bigl(-\Delta_{\bx_1} - \Delta_{\bx_2}
  + V_{\mathrm{trap}}(\bx_1) + V_{\mathrm{trap}}(\bx_2)
  + \tfrac{c}{\sqrt{|\bx_1 - \bx_2|^2 + \delta^2}}\bigr)\psi,
  \quad \bx_1, \bx_2 \in \R^2,
\end{equation}
where $V_{\mathrm{trap}}(\bx) = |\bx|^2$ is a harmonic trap.

\subsubsection{Preconditioner comparison}

To select the best preconditioner for PCG, we solve
$(-\Delta + V_1 + V_2)\,u = f$ with a random right-hand side $f$
on the $49^4$ grid with $\delta = 0.01$, comparing three
preconditioners in Table~\ref{tab:pcg_precond}.
The tensor-product preconditioner
$(-\Delta + V_1)^{-1}$ converges in 24 iterations,
far fewer than the combined preconditioner 
$P_C = (V_1{+}V_2)^{-1/2}(-\Delta)^{-1}(V_1{+}V_2)^{-1/2}$
(443 iterations).
The variant $V_2^{-1/2}(-\Delta{+}V_1)^{-1}V_2^{-1/2}$ does not
converge within 500 iterations.
We therefore use $(-\Delta + V_1)^{-1}$ throughout.

\begin{table}[ht]\centering
\caption{PCG iteration count for solving
$(-\Delta + V_1 + V_2)\,u = f$ with three preconditioners.
$Q^{10}$ SEM, $49^4$ DoFs, $\delta = 0.01$, FP64, A100.
Tolerance $10^{-12}$, max 500 iterations.}
\label{tab:pcg_precond}
{\small
\begin{tabular}{@{}lrrl@{}}
\toprule
Preconditioner & Iters & Time (s) & Status \\
\midrule
$(-\Delta + V_1)^{-1}$ & 24 & 0.5 & converged \\
$(V_1{+}V_2)^{-1/2}(-\Delta)^{-1}(V_1{+}V_2)^{-1/2}$ & 443 & 2.9 & converged \\
$V_2^{-1/2}(-\Delta{+}V_1)^{-1}V_2^{-1/2}$ & 500 & 3.0 & stalled at $5 \times 10^{-10}$ \\
\bottomrule
\end{tabular}
}
\end{table}

\subsubsection{Ground state computation}

Table~\ref{tab:coulomb_ground_99}
shows the ground state computation cost for several values
of $\delta$, with $c = 1$, $V_{\mathrm{trap}}(\bx) = |\bx|^2$,
domain $[-8,8]^2 \times [-8,8]^2$.
PCG is warm-started with the previous inverse iteration solution,
which reduces iterations after the first few steps.

\begin{table}[ht]\centering
\caption{Ground state of $H = -\Delta + V_{\mathrm{trap}} + V_{\rm Coulomb}$
via shifted inverse iteration with PCG
(preconditioner $(-\Delta + V_{\mathrm{trap}})^{-1}$, max $500$ iterations,
warm-started with the previous iteration's solution).
$Q^{10}$ SEM, $99^4 \approx 9.6 \times 10^7$ DoFs, GH200.
Multi-level: coarse $49^4$ solution interpolated as initial guess;
shift $\sigma = \lambda_{\min}(-\Delta+V_{\mathrm{trap}}) - 10^{-4} = 3.9999$.
PCG tolerance $10^{-9}$.}
\label{tab:coulomb_ground_99}
{\small
\begin{tabular}{@{}rrrrrrl@{}}
\toprule
$\delta$ & $V_{\rm C}^{\max}$ & Inv.\ iters & PCG/iter & Total solves & Time (s) & $\lambda_1$ \\
\midrule
0.1    & 10      & 22 & 1--18       & 121         & 8      & 5.060514417326 \\
0.01   & 100     & 24 & 1--32       & 204         & 12     & 5.221829172892 \\
0.001  & 1{,}000 & 12 & {1--500$^\dagger$}  & 4{,}286     & 208    & 5.400105827385 \\
0.0001 & 10{,}000 & 12 & {1--500$^\dagger$} & 5{,}816     & 274    & 5.444872472429 \\
\bottomrule
\end{tabular}
}

\medskip\footnotesize
$\dagger$\, PCG reaches max 500 iterations only in the first few inverse iterations, which does not affect the final accuracy since PCG converges quickly in the last few inverse iterations.  
\end{table}

Figure~\ref{fig:coulomb4d_slices} shows slices of the 4D ground state
wavefunction for $\delta = 0.01$ and $0.0001$,
interpolated onto a fine grid via cell-by-cell
$Q^{10}$ polynomial reconstruction.
As $\delta$ decreases, the Coulomb singularity sharpens:
the 1D slice narrows, and the inter-particle correlation plot
develops a more pronounced dip along $x_{1a} = x_{2a}$.

\begin{figure}[ht]\centering
\includegraphics[width=\textwidth]{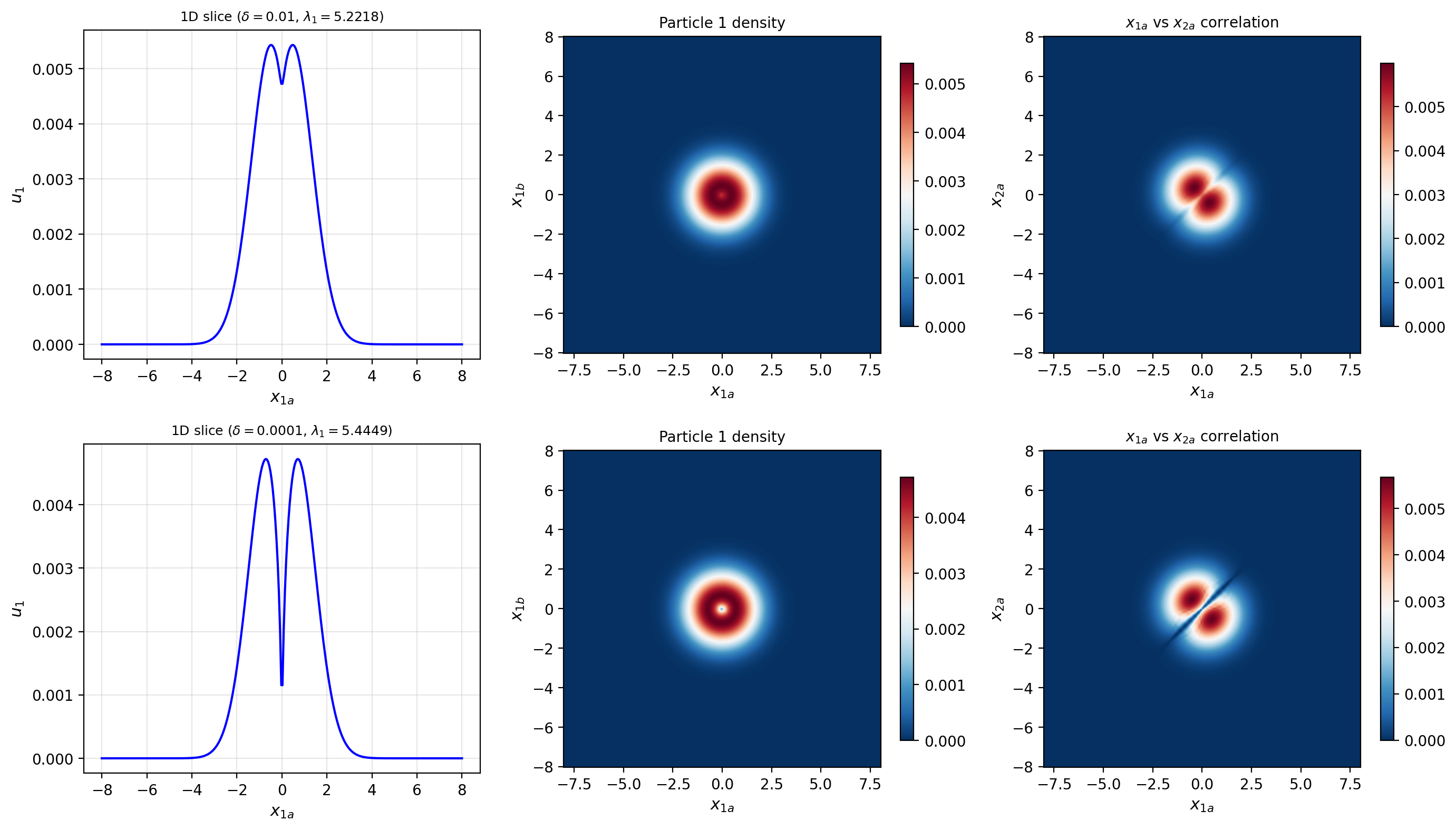}
\caption{Slices of the 4D Coulomb ground state for
$\delta = 0.01$ (top) and $\delta = 0.0001$ (bottom),
$Q^{10}$ SEM, $99^4$ DoFs.
Top: 1D slice along $x_{1a}$.
Middle: particle~1 density in the $(x_{1a}, x_{1b})$ plane.
Bottom: $x_{1a}$ vs.\ $x_{2a}$ inter-particle correlation.
The Coulomb repulsion strengthens as $\delta \to 0$.
Plotted on a $300 \times 300$ fine grid via cell-by-cell
$Q^{10}$ polynomial interpolation.}
\label{fig:coulomb4d_slices}
\end{figure}

\subsubsection{Splitting convergence}
 
As in Section~\ref{sec:ham_separable}, ``qHOP'' and ``Magnus-2''
refer to the approximated versions~\eqref{eq:qhop_product}
and~\eqref{eq:yoshida}.
All 4D splitting tests use $Q^{10}$ SEM, $99^4$ DoFs,
complex128, $T = 0.1$, GH200.

\begin{example}[Order of approximated qHOP convergence in 4D, $\delta = 0.001$]
\label{ex:coulomb_qhop_d001}
Table~\ref{tab:coulomb_qhop_d001} reports qHOP with $M = 1$ (Strang), $3$, $5$, $7$
for the Coulomb system with $\delta = 0.001$ ($V_{\rm C}^{\max} = 1000$).
All methods show $O(\Delta t^2)$ convergence in the asymptotic regime
($\Delta t \le 0.001$).
Larger $M$ reduces the error constant:
at $\Delta t = 0.0005$, $M{=}3$, $5$, $7$ reduce the error by
${\sim}14{\times}$, $35{\times}$, $65{\times}$ relative to Strang.
\end{example}

\begin{table}[ht]\centering
\caption{qHOP splitting error
$\|\psi_{\rm split}(T){-}\psi_{\rm exact}(T)\|_2$
on GH200,
$Q^{10}$ SEM, $99^4$ DoFs, complex128, $T{=}0.1$,
Coulomb $\delta = 0.001$. $t$\,(s) is the GPU time.}
\label{tab:coulomb_qhop_d001}
{\small
\begin{tabular*}{\textwidth}{@{\extracolsep{\fill}}rrrrrrrrrrrrrr@{}}
\toprule
& \multicolumn{3}{c}{$M=1$ (Strang)} & \multicolumn{3}{c}{$M=3$} & \multicolumn{3}{c}{$M=5$} & \multicolumn{3}{c}{$M=7$} \\
\cmidrule(lr){2-4}\cmidrule(lr){5-7}\cmidrule(lr){8-10}\cmidrule(lr){11-13}
$\Delta t$ & Error & Rate & $t$ (s) & Error & Rate & $t$ (s) & Error & Rate & $t$ (s) & Error & Rate & $t$ (s) \\
\midrule
0.005  & 3.32e-1 & --   & 1.8 & 4.93e-2 & --   & 4.2 & 4.27e-2 & --   & 5.9 & 2.96e-2 & --   & 8.0 \\
0.002  & 5.21e-2 & 2.02 & 3.3 & 2.51e-2 & 0.74 & 8.5 & 1.26e-3 & 3.84 & 13  & 5.18e-4 & 4.42 & 18  \\
0.001  & 3.75e-2 & 0.47 & 5.9 & 1.33e-3 & 4.24 & 16  & 2.34e-4 & 2.43 & 26  & 1.25e-4 & 2.05 & 36  \\
0.0005 & 1.98e-3 & 4.24 & 11  & 1.41e-4 & 3.24 & 31  & 5.64e-5 & 2.05 & 51  & 3.03e-5 & 2.04 & 71  \\
\bottomrule
\end{tabular*}
}
\end{table}

\begin{example}[Order of approximated Magnus-2 in 4D, $\delta = 0.1$]
\label{ex:coulomb_magnus2_d01}
Table~\ref{tab:coulomb_magnus2_d01} shows the convergence of Magnus-2 via Yoshida with $M = 1$ (plain Yoshida),
$3$, $5$, $7$ for the Coulomb system with $\delta = 0.1$
($V_{\rm C}^{\max} = 10$).
For $M = 7$, the error reaches $2.5 \times 10^{-10}$ at $\Delta t = 0.001$.
Larger $M$   reduces the error constant and
enters the high-order asymptotic regime at larger $\Delta t$.
\end{example}

\begin{table}[ht]\centering
\caption{Magnus-2 via Yoshida splitting error,
$99^4$ DoFs, complex128, $T{=}0.1$,
Coulomb $\delta = 0.1$. $t$\,(s) is the GPU time on GH200.}
\label{tab:coulomb_magnus2_d01}
{\small
\begin{tabular*}{\textwidth}{@{\extracolsep{\fill}}rrrrrrrrrrrrrr@{}}
\toprule
& \multicolumn{3}{c}{$M=1$ (Yoshida)} & \multicolumn{3}{c}{$M=3$} & \multicolumn{3}{c}{$M=5$} & \multicolumn{3}{c}{$M=7$} \\
\cmidrule(lr){2-4}\cmidrule(lr){5-7}\cmidrule(lr){8-10}\cmidrule(lr){11-13}
$\Delta t$ & Error & Rate & $t$ (s) & Error & Rate & $t$ (s) & Error & Rate & $t$ (s) & Error & Rate & $t$ (s) \\
\midrule
0.005 & 2.54e-3 & --   & 3.7 & 7.49e-4 & --   & 10 & 2.83e-4 & --   & 16 & 1.51e-4 & --    & 22  \\
0.002 & 5.82e-5 & 4.12 & 8.4 & 2.38e-5 & 3.76 & 24 & 5.36e-6 & 4.33 & 39 & 5.52e-7 & 6.12  & 53  \\
0.001 & 8.39e-6 & 2.79 & 16  & 1.61e-6 & 3.89 & 46 & 3.74e-8 & 7.16 & 77 & 2.54e-10 & 11.09 & 106 \\
\bottomrule
\end{tabular*}
}
\end{table}

\subsection{Two particles in 3D with Coulomb interaction (6D)}
\label{sec:coulomb6d}

We extend the 4D test to two particles in 3D
($\boldsymbol{x}_1, \boldsymbol{x}_2 \in \R^3$).
The Hamiltonian is
\begin{equation}\label{eq:coulomb6d}
  H = -\Delta_{\boldsymbol{x}_1} - \Delta_{\boldsymbol{x}_2}
    + V_{\mathrm{trap}}(\boldsymbol{x}_1) + V_{\mathrm{trap}}(\boldsymbol{x}_2)
    + \frac{c}{\sqrt{|\boldsymbol{x}_1 - \boldsymbol{x}_2|^2 + \delta^2}},
\end{equation}
discretized on $[-L,L]^3 \times [-L,L]^3$ with $Q^{10}$ SEM, $29^6 \approx 5.9 \times 10^8$ DoFs.
The splitting is $A = -\Delta$ (6D Laplacian),
$B = V_{\mathrm{trap}} + V_{\mathrm{Coulomb}}$,
and the reference solution is
$\psi_{\mathrm{exact}}(\boldsymbol{x}, t) = e^{-i\lambda_1 t}\,u_1(\boldsymbol{x})$
where $(\lambda_1, u_1)$ is the ground state of $H$.

\subsubsection{Ground state}

In these TF32 runs, the offline eigendecomposition of
$-\Delta + V_{\mathrm{trap}}$ is still computed in FP64 and then cast to
lower precision for  PCG solves. See Table \ref{tab:coulomb6d_ground_tf32} for the performance. The solution is plotted in Figure~\ref{fig:coulomb6d_tf32}.

\begin{figure}[ht]\centering
\includegraphics[width=\textwidth]{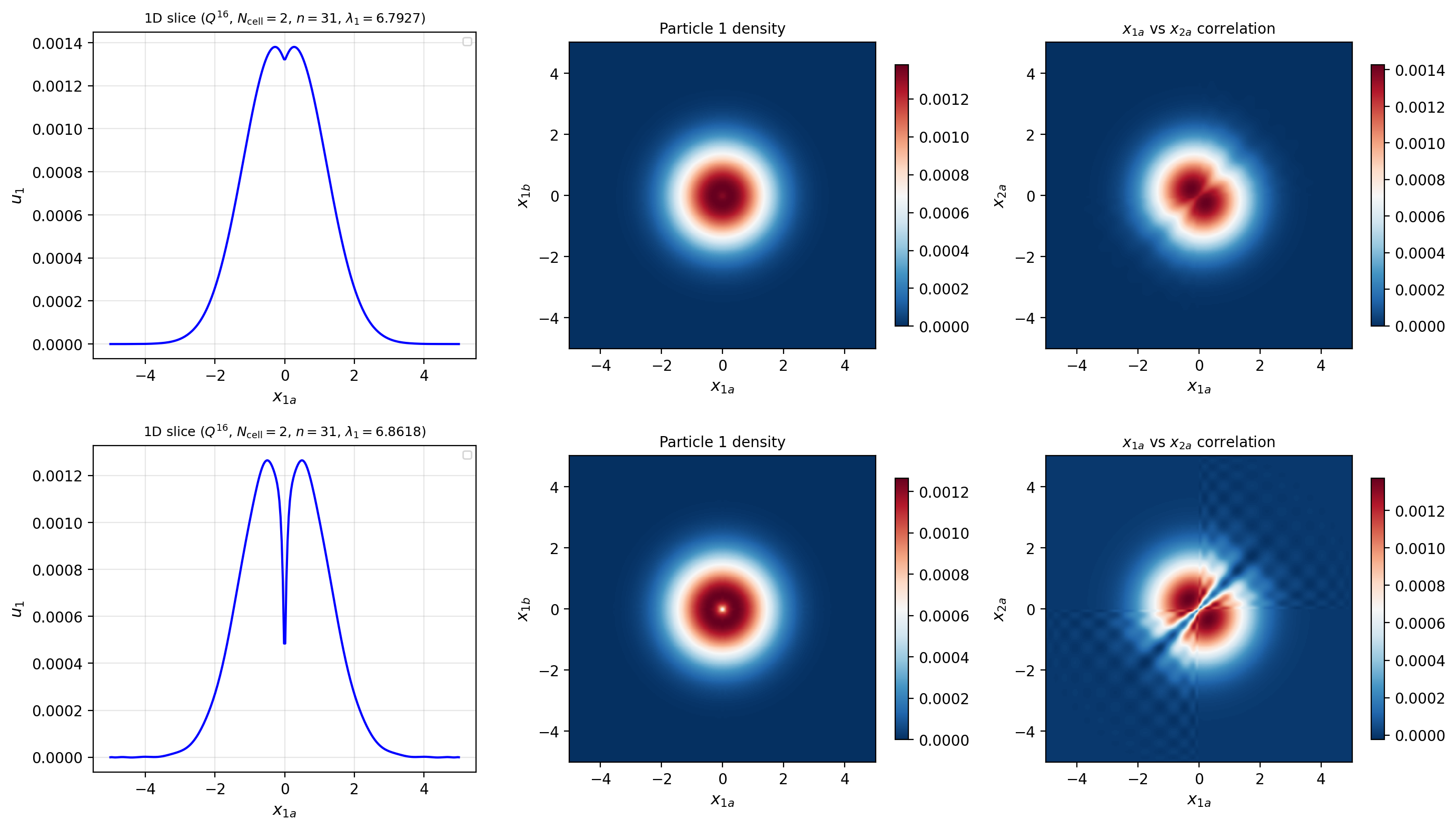}
\caption{Slices of the 6D Coulomb ground state at $31^6 \approx 8.9 \times 10^8$ DoFs,
$Q^{16}$ SEM, $N_{\rm cell} = 2$, TF32, $L = 5$, GH200.
Top: $\delta = 0.01$. Bottom: $\delta = 0.0001$.
Left: 1D slice along $x_{1a}$.
Center: particle~1 density.
Right: inter-particle correlation $x_{1a}$ vs.\ $x_{2a}$.
Plotted on a $300 \times 300$ fine grid via cell-by-cell
$Q^{16}$ polynomial interpolation.}
\label{fig:coulomb6d_tf32}
\end{figure}

\begin{table}[ht]\centering
\caption{Ground state of the 6D Coulomb Hamiltonian~\eqref{eq:coulomb6d}, $Q^{16}$ SEM, $N_{\rm cell} = 2$, $n = 31$,
$31^6 \approx 8.9 \times 10^8$ DoFs, TF32, $L = 5$, GH200.
Shifted inverse iteration with $\sigma = 0$,
PCG tolerance $10^{-4}$.}
\label{tab:coulomb6d_ground_tf32}
{\small
\begin{tabular}{@{}rrrrrr@{}}
\toprule
$\delta$ & $V_{\rm C}^{\max}$ & Inv.\ iters & Total PCG solves & Time (s) & $\lambda_1$ \\
\midrule
0.01   & 100   & 12 &  31 &  17 & 6.792689 \\
0.001  & 1000  & 12 &  99 &  34 & 6.844492 \\
0.0001 & 10000 & 12 & 331 &  94 & 6.861799 \\
\bottomrule
\end{tabular}
}
\end{table}

\begin{table}[ht]\centering
\caption{qHOP splitting error for the 6D Coulomb
system~\eqref{eq:coulomb6d}, $\delta = 0.01$,
$Q^{10}$ SEM, $29^6$ DoFs, complex64, $T{=}0.1$, GH200. $t$\,(s) is the GPU time.}
\label{tab:coulomb6d_qhop}
{\small
\begin{tabular*}{\textwidth}{@{\extracolsep{\fill}}rrrrrrrrrrrrrr@{}}
\toprule
& \multicolumn{3}{c}{$M=1$ (Strang)} & \multicolumn{3}{c}{$M=3$} & \multicolumn{3}{c}{$M=5$} & \multicolumn{3}{c}{$M=7$} \\
\cmidrule(lr){2-4}\cmidrule(lr){5-7}\cmidrule(lr){8-10}\cmidrule(lr){11-13}
$\Delta t$ & Error & Rate & $t$ (s) & Error & Rate & $t$ (s) & Error & Rate & $t$ (s) & Error & Rate & $t$ (s) \\
\midrule
0.1   & 1.98e-1 & --   & 1.5 & 1.29e-1 & --   & 2.2 & 1.75e-2 & --   & 3.5 & 8.34e-3 & --   & 3.6 \\
0.02  & 9.76e-3 & 1.87 & 2.9 & 1.22e-3 & 2.90 & 5.8 & 4.87e-4 & 2.23 & 8.7 & 2.61e-4 & 2.15 & 12  \\
0.005 & 5.51e-4 & 2.07 & 8.0 & 7.43e-5 & 2.02 & 19  & 3.00e-5 & 2.01 & 31  & 1.62e-5 & 2.01 & 43  \\
\bottomrule
\end{tabular*}
}
\end{table}

\subsubsection{Splitting convergence}

\begin{example}[Order of approximated qHOP in 6D, $\delta = 0.01$]
\label{ex:coulomb6d_qhop}
Table~\ref{tab:coulomb6d_qhop} reports qHOP with $M = 1$ (Strang), $3$, $5$, $7$
for the 6D Coulomb system with $\delta = 0.01$.
All methods show $O(\Delta t^2)$ convergence.
Larger $M$ reduces the error constant by
${\sim}7{\times}$ ($M{=}3$), $15{\times}$ ($M{=}5$),
$18{\times}$ ($M{=}7$) relative to Strang.
The splitting tests use complex64 arithmetic to fit the $29^6$
arrays in GPU memory, and the reference ground state is computed in FP64.
\end{example}

\begin{example}[Order of approximated Magnus-2 in 6D, $\delta = 0.01$]
\label{ex:coulomb6d_magnus2}
Table~\ref{tab:coulomb6d_magnus2} reports Magnus-2 via Yoshida with $M = 1$, $3$, $5$, $7$.
At large $\Delta t$, larger $M$ reduces the error:
at $\Delta t = 0.01$, the errors range from $7.93 \times 10^{-4}$ ($M{=}1$)
to $3.28 \times 10^{-6}$ ($M{=}7$).
At smaller $\Delta t$, the errors plateau at ${\sim}2.5 \times 10^{-6}$
due to the complex64 arithmetic floor: with 20--50 time steps, single-precision
rounding errors accumulate to $O(10^{-6})$. 
\end{example}

\begin{table}[ht]\centering
\caption{Magnus-2 via Yoshida splitting error for the
6D Coulomb system~\eqref{eq:coulomb6d}, $\delta = 0.01$,
$Q^{10}$ SEM, $29^6$ DoFs, complex64, $T{=}0.1$, GH200. $t$\,(s) is the GPU time.}
\label{tab:coulomb6d_magnus2}
{\small
\begin{tabular*}{\textwidth}{@{\extracolsep{\fill}}rrrrrrrrrrrrrr@{}}
\toprule
& \multicolumn{3}{c}{$M=1$ (Yoshida)} & \multicolumn{3}{c}{$M=3$} & \multicolumn{3}{c}{$M=5$} & \multicolumn{3}{c}{$M=7$} \\
\cmidrule(lr){2-4}\cmidrule(lr){5-7}\cmidrule(lr){8-10}\cmidrule(lr){11-13}
$\Delta t$ & Error & Rate & $t$ (s) & Error & Rate & $t$ (s) & Error & Rate & $t$ (s) & Error & Rate & $t$ (s) \\
\midrule
0.1   & 2.43e-1 & --   & 2.3 & 1.78e-1 & --   & 4.4 & 2.52e-1 & --   & 6.0 & 2.72e-2 & --    & 7.7  \\
0.01  & 7.93e-4 & 2.49 & 12  & 1.55e-5 & 4.06 & 28  & 3.29e-6 & 4.88 & 46  & 1.99e-6 & 4.14  & 65   \\
0.005 & 5.22e-5 & 3.93 & 22  & 2.34e-6$^\dagger$ & 2.73 & 55  & 1.66e-6$^\dagger$ & 0.99 & 91  & 2.77e-6$^\dagger$ & $-0.48$ & 128 \\
\bottomrule
\end{tabular*}
}

\medskip\footnotesize
$^\dagger$Error dominated by complex64 rounding accumulation.
\end{table}

\subsection{Three particles in 3D with Coulomb interaction (9D)}
\label{sec:fewbody}

We consider three particles in 3D
($\boldsymbol{x}_1, \boldsymbol{x}_2, \boldsymbol{x}_3 \in \R^3$)
in a harmonic trap with pairwise soft Coulomb interaction:
\begin{equation}\label{eq:fewbody}
  H = \sum_{j=1}^{3}\bigl(-\Delta_{\boldsymbol{x}_j}
      + V_{\mathrm{trap}}(\boldsymbol{x}_j)\bigr)
    + \sum_{1 \le j < k \le 3}
      \frac{c}{\sqrt{|\boldsymbol{x}_j - \boldsymbol{x}_k|^2 + \delta^2}},
\end{equation}
where $V_{\mathrm{trap}}(\boldsymbol{x}) = |\boldsymbol{x}|^2$.
Discretized on $[-L,L]^3 \times [-L,L]^3 \times [-L,L]^3$ with
$Q^k$ SEM, this gives $n^9$ DoFs.

\subsubsection{Ground state}

The offline eigendecomposition used in the separable preconditioner is
computed in FP64, even though the online inverse iteration and PCG
solves are carried out in FP32.
The ground state is computed by shifted inverse iteration with
$\sigma = 0$ and warm-start PCG
(preconditioner $\bigl(\sum_j(-\Delta_{\boldsymbol{x}_j} + V_{\mathrm{trap}})\bigr)^{-1}$,
tolerance $10^{-4}$ with max $500$ iterations).
Warm-start uses the previous outer iteration's PCG solution as the
initial guess for the next solve. 
Table~\ref{tab:fewbody_coulomb} shows the ground state cost
for $L = 3$, $c = 1$, $Q^5$ SEM, $N_{\rm cell} = 2$, $n = 9$
($9^9 \approx 3.9 \times 10^8$ DoFs) on GH200.
PCG iterations increase with  smaller $\delta$,
but all cases converge within 18 inverse iterations.

\begin{table}[ht]\centering
\caption{Ground state of the 9D Coulomb
Hamiltonian~\eqref{eq:fewbody},
3 particles in 3D,
$Q^5$ SEM, $N_{\rm cell} = 2$, $9^9 \approx 3.9 \times 10^8$ DoFs,
FP32, $\sigma = 0$, $L = 3$, GH200.}
\label{tab:fewbody_coulomb}
{\small
\begin{tabular}{@{}rrrrrr@{}}
\toprule
$\delta$ & $V_{\rm C}^{\max}$ & Inv.\ iters & PCG solves & Time (s) & $\lambda_1$ \\
\midrule
0.01   & 100   & 18 & 110 & 490 & 11.7561 \\
0.001  & 1000  & 18 & 349 & 588 & 11.9268 \\
0.0001 & 10000 & 18 & 708 & 714 & 11.9502 \\
\bottomrule
\end{tabular}
}
\end{table}

Figure~\ref{fig:fewbody_coulomb} shows slices of the ground state
for two values of $\delta$.
As $\delta$ decreases, the Coulomb repulsion sharpens and the
inter-particle anti-correlation pattern along $x_{1a} = x_{2a}$
becomes more pronounced. 

\begin{figure}[ht]\centering
\includegraphics[width=\textwidth]{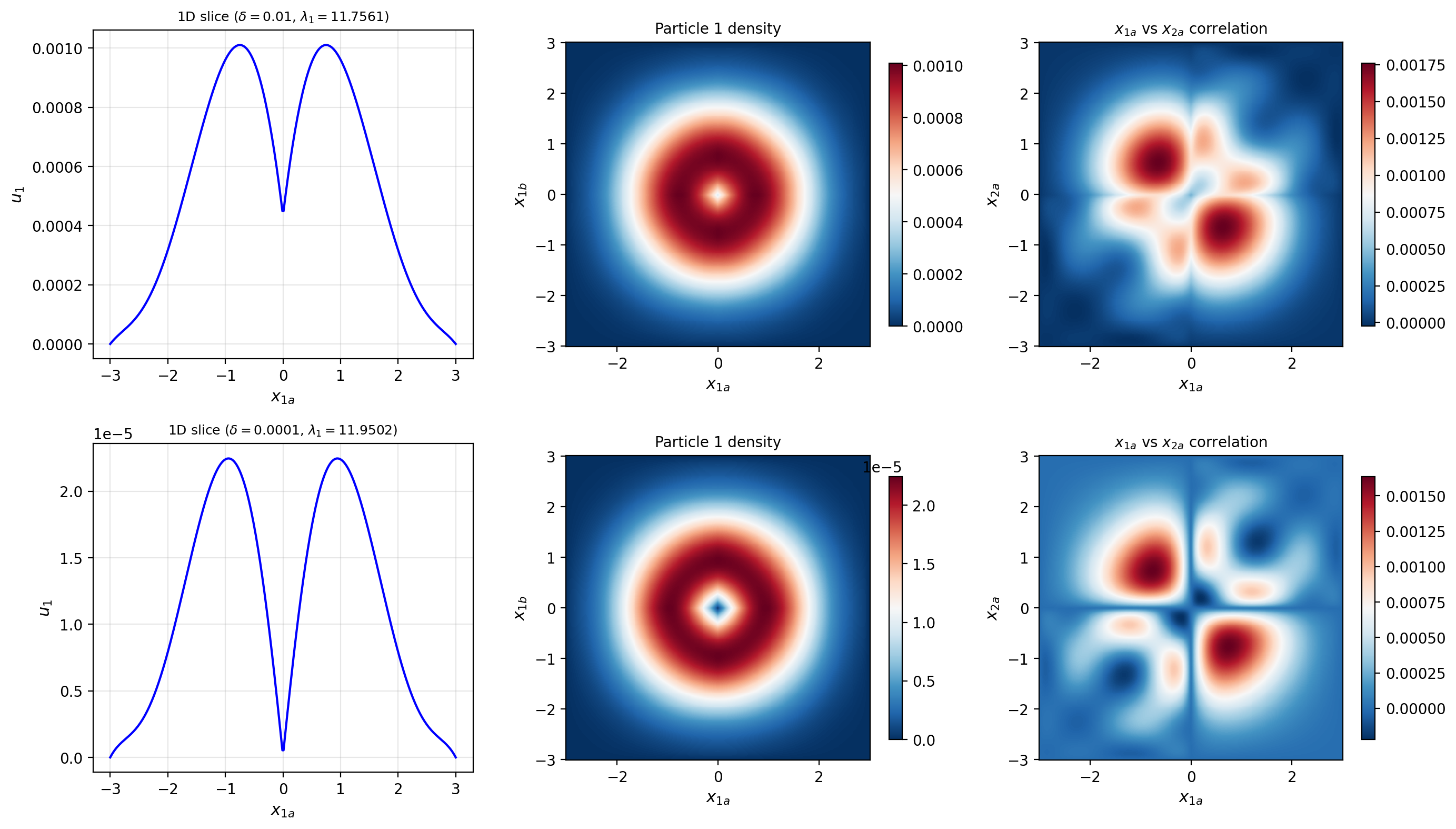}
\caption{Slices of the 9D Coulomb ground state,
3 particles in 3D, $Q^5$ SEM, $9^9$ DoFs, FP32, $L = 3$, GH200.
Top: $\delta = 0.01$. Bottom: $\delta = 0.0001$.
Left: 1D slice along $x_{1a}$.
Center: particle~1 density in the $(x_{1a}, x_{1b})$ plane.
Right: $x_{1a}$ vs.\ $x_{2a}$ inter-particle correlation.
Plotted on a $300 \times 300$ fine grid via cell-by-cell
$Q^5$ polynomial interpolation.}
\label{fig:fewbody_coulomb}
\end{figure}

\subsubsection{Splitting convergence} 
The splitting $A = \sum_j(-\Delta_{\boldsymbol{x}_j} + V_{\mathrm{trap}})$,
$B = V_{12} + V_{13} + V_{23}$ is tested with $\delta = 0.1$,
$T = 0.1$, complex64 arrays, using $Q^5$ SEM with $N_{\mathrm{cell}} = 2$
($n = 9$, $9^9 \approx 3.9 \times 10^8$ DoFs) on GH200.
Errors of qHOP and Magnus-2 via Yoshida for $M = 1, 3, 5, 7$
are shown in Figure~\ref{fig:fewbody_coulomb_splitting}. 

\begin{figure}[ht]\centering
\begin{minipage}[t]{0.48\textwidth}\centering
\includegraphics[width=\textwidth]{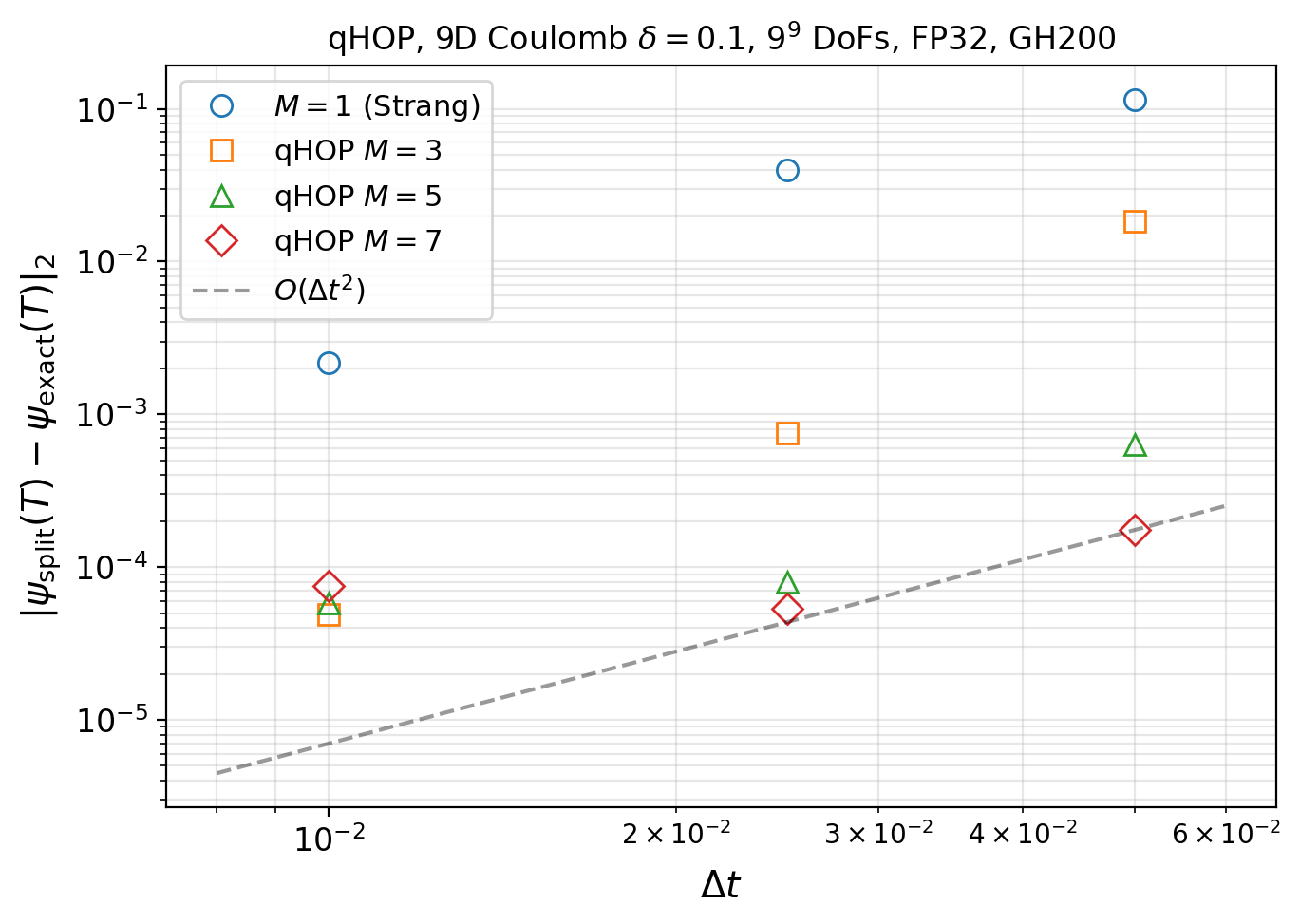}
\end{minipage}\hfill
\begin{minipage}[t]{0.48\textwidth}\centering
\includegraphics[width=\textwidth]{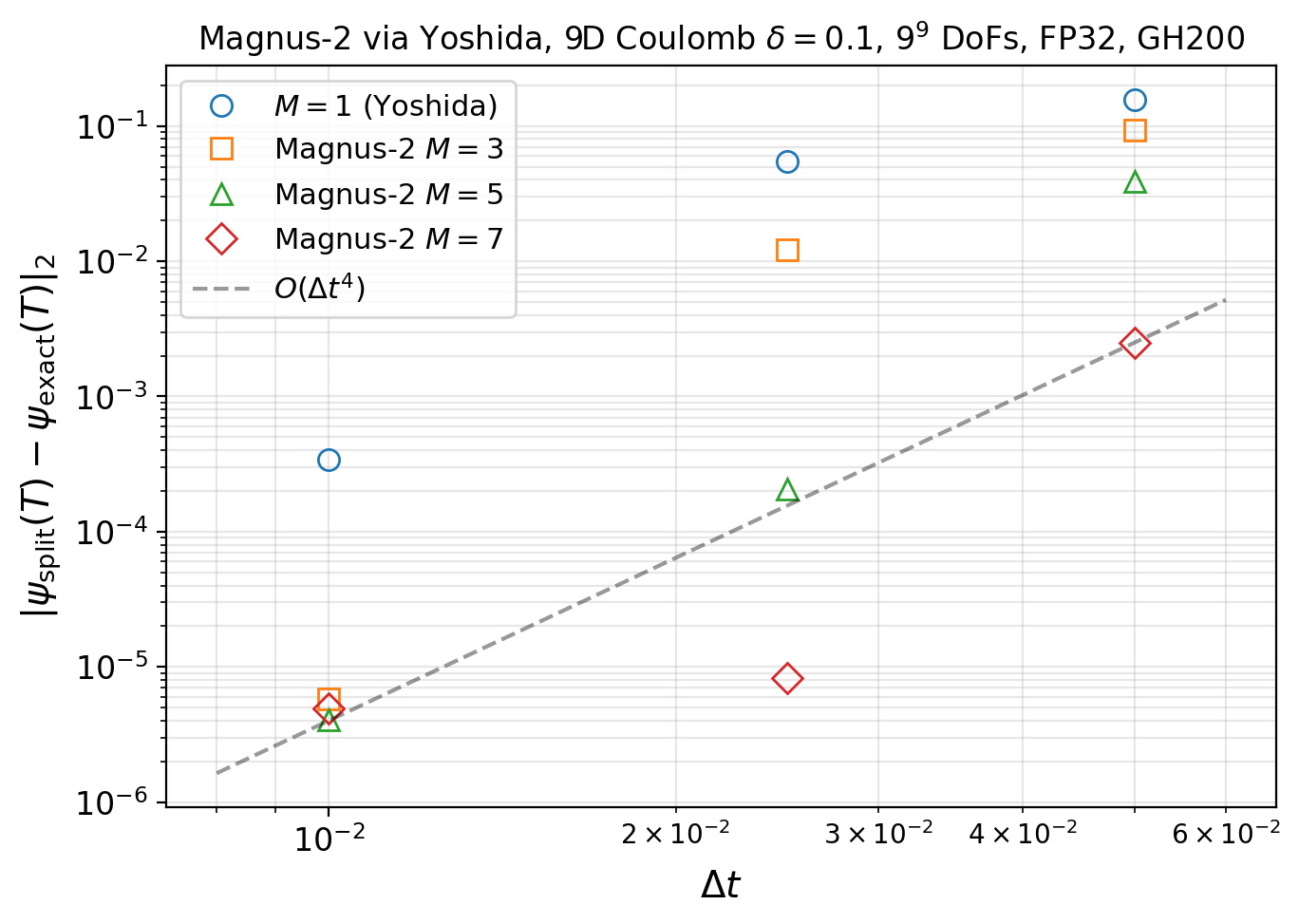}
\end{minipage}
\caption{Splitting convergence for 9D Coulomb, $\delta = 0.1$,
$9^9 \approx 3.9 \times 10^8$ DoFs, complex64, FP64 transforms,
$T = 0.1$, GH200.
Left: qHOP ($O(\Delta t^2)$).
Right: Magnus-2 via Yoshida ($O(\Delta t^4)$).
Larger $M$ reduces the error constant, and all methods hit the
complex64 floor at ${\sim}5 \times 10^{-6}$.}
\label{fig:fewbody_coulomb_splitting}
\end{figure}

\section{Concluding remarks}
\label{sec:conclusion}

We extended the tensor-product solver of~\cite{ZhangShenLiu}
from the Laplacian to $-\Delta + V$ on a single GPU, handling up to
$10^9$ DoFs in under one second.
For non-separable potentials $V = V_1 + V_2$, the preconditioner
$(-\Delta + V_1)^{-1}$ yields a preconditioned operator with 
a bounded condition number and spectrum clustering near~$1$, independently of the mesh size
(Theorem~\ref{thm:pcg}), and also independently of the domain
size when $V_1$ is confining and $V_2$ is bounded
(Theorem~\ref{cor:confining}).
We applied this framework to ground state computation via shifted
inverse iteration, PCG, and Gross--Pitaevskii gradient flows,
and used the resulting eigenpairs as exact stationary
solutions to validate the approximated versions of qHOP and Magnus-2 splitting methods
at ${\sim}10^8$ DoFs in 3D through 9D.
The product-formula implementation showed the expected
convergence orders  and also provides a family
of classical splitting methods generalizing second order Strang and fourth order Yoshida splittings.

\section*{Acknowledgments}
The authors are grateful to Prof. Antoine Levitt at Laboratoire de Math\'ematiques d'Orsay,
  Universit\'e Paris-Saclay and Prof. Jianfeng Lu at Duke University for discussions on the Schr\"odinger operator.
    This work was partially supported by NSF DMS-2208518.
This work used DeltaAI at the National Center for Supercomputing
Applications (NCSA) through allocation MTH260013 from the Advanced
Cyberinfrastructure Coordination Ecosystem: Services \& Support
(ACCESS) program, which is supported by U.S.\ National Science
Foundation grants \#2138259, \#2138286, \#2138307, \#2137603,
and \#2138296.

\section*{Declaration of AI-assisted technologies}
All original ideas, as well as the direct solver and PCG codes, are attributed to the authors.
During the preparation of this work, the authors used Anthropic’s Claude Code to assist with mathematical discussions, further numerical implementation, and drafting of the manuscript. After using this tool, the authors carefully reviewed and edited the content as needed and take full responsibility for the final publication.
\appendix

\section{Hermite spectral method on unbounded domains}
\label{app:hermite}

The tensor-product solver of Section~\ref{sec:solver} applies to
\emph{any} discretization that produces a separable 1D operator
$K = T\Lambda T^{-1}$ per axis. To illustrate this generality, we present the Hermite spectral method, which replaces the SEM on $[-L,L]^d$ with Hermite functions on $\R^d$, eliminating domain truncation artifacts for quantum systems in confining potentials.
The $n$ Hermite-Gauss nodes $\{x_j\}$ are eigenvalues of the $n \times n$ symmetric tridiagonal matrix $J$ with $J_{k,k} = 0$ and $J_{k,k+1} = J_{k+1,k} = \sqrt{k/2}$ for $k = 1, \ldots, n-1$. The Hermite function differentiation matrix at these nodes can be written as
\[
  D_{ij} = \frac{\psi_{n-1}(x_i)}
  {\psi_{n-1}(x_j)\,(x_i - x_j)}, \quad i \ne j,
  \qquad D_{ii} = 0,
\]
where $\psi_{n-1}$ is the normalized Hermite function of degree $n-1$. See~\cite[Eq.~(7.93)]{ShenTangWang2011}.
Given a separable potential $f(x)$, the 1D operator is
$K = -D^2 + \mathrm{diag}(f(x_j))$,
with eigendecomposition $K = T\Lambda T^{-1}$.
The $d$-dimensional tensor-product solver is identical to SEM.
No boundary conditions are needed, since the Hermite functions decay
as $e^{-x^2/2}$.
The method is limited to $n \lesssim 745$ in FP64 due to
underflow in the Hermite function recurrence at the outermost nodes.
Table~\ref{tab:hermite_accuracy} shows the accuracy of
$(-\Delta + V)^{-1}$ on $\R^3$ with the potential~\eqref{eq:test_potential}. 
Figure~\ref{fig:pcg_hermite} shows that the proposed preconditioner
$(-\Delta + V_1)^{-1}$ converges in 5 PCG iterations on $\R^3$,
while the baseline $P_C$ does not converge within 500 iterations.
Table~\ref{tab:magnus2_hermite} reports the Magnus-2 convergence
using the Hermite spectral method with separable
potential~\eqref{eq:separable_V1}, $499^3$ DoFs, $T = 1$.
For $M \ge 3$, the observed convergence is $O(\Delta t^4)$. 

\begin{table}[ht]\centering
\caption{Accuracy for using $(-\Delta+V)^{-1}$ to solve $(-\Delta + V)u=f$ with  potential~\eqref{eq:test_potential} on $\R^3$ via Hermite spectral
method, FP64, A100. The exact solution is $u^*= \frac{\sin\big(\tfrac{\pi}{2}(x+1)\big)}{1+x^2}\cdot\frac{\sin\big(\pi
  (y+1)\big)}{1+y^2}\cdot\frac{\sin\big(\tfrac{3\pi}{2}(z+1)\big)}{1+z^2}$.
}
\label{tab:hermite_accuracy}
{\small
\begin{tabular}{@{}rrrrrr@{}}
\toprule
$n$ & DoFs & Setup (s) & Solve (s) & Weighted $L^2$ err & $\ell^2$ rel.\ err \\
\midrule
199 & $199^3$ & 2.7 & 0.002 & $1.53\times10^{-5}$ & $6.31\times10^{-5}$  \\
399 & $399^3$ & 3.2 & 0.024 & $4.05\times10^{-6}$ & $2.20\times10^{-5}$  \\
599 & $599^3$ & 3.2 & 0.117 & $1.87\times10^{-6}$ & $1.20\times10^{-5}$  \\
\bottomrule
\end{tabular}
}
\end{table}

\begin{figure}[ht]\centering
\includegraphics[width=0.48\textwidth]{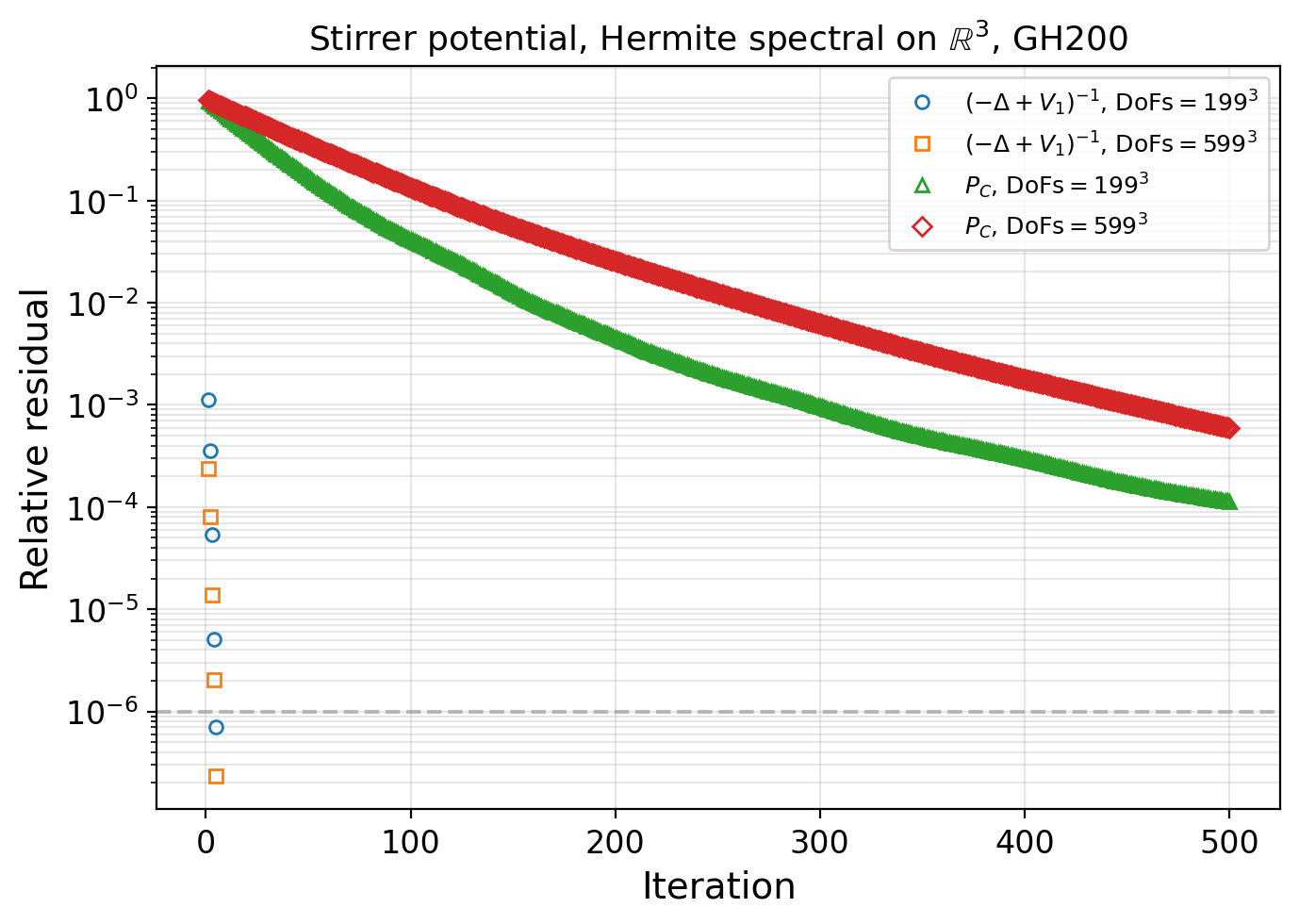}
\hfill
\includegraphics[width=0.48\textwidth]{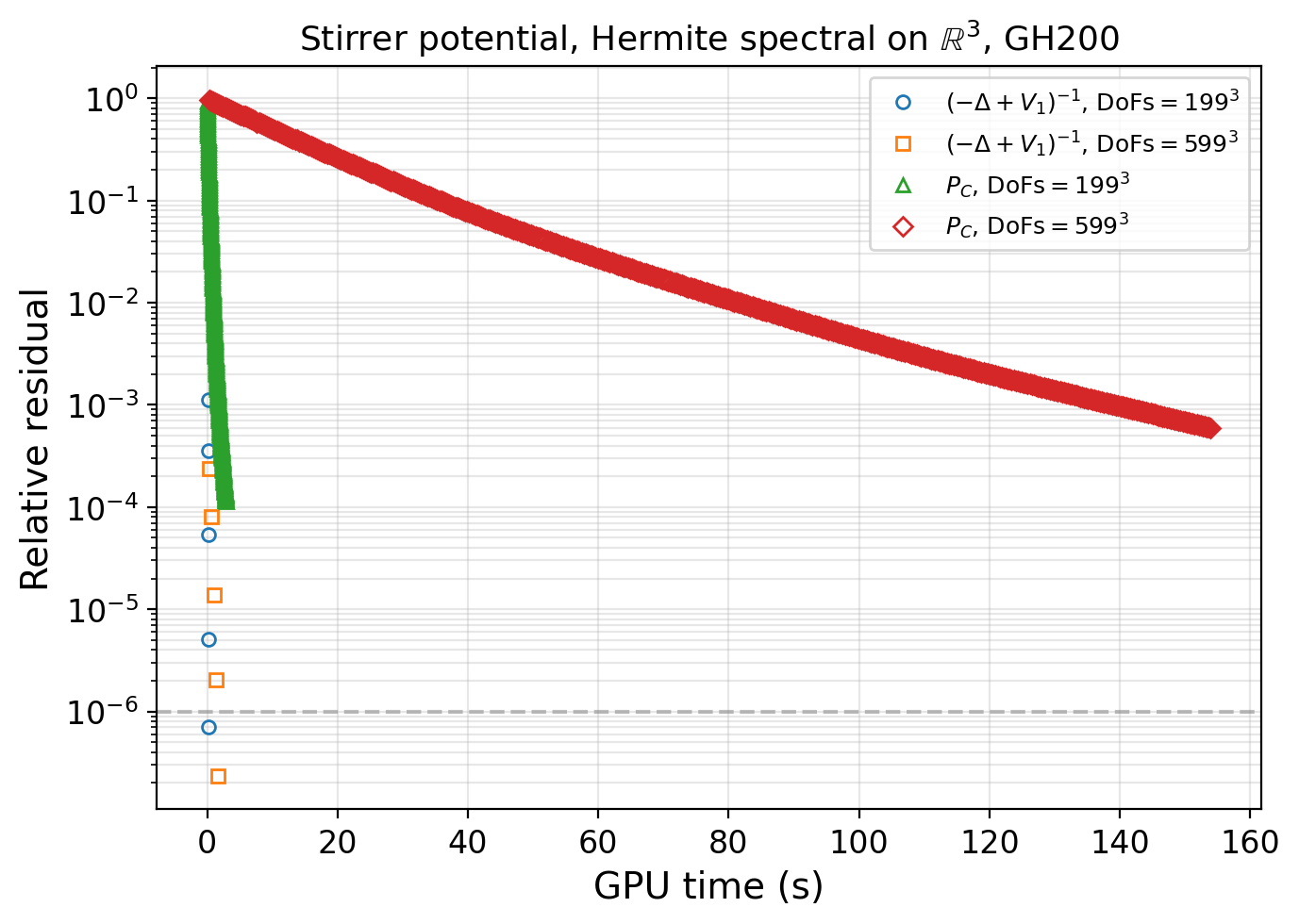}
\caption{PCG convergence with Hermite spectral method on $\R^3$.
Stirrer potential, FP32, GH200.}
\label{fig:pcg_hermite}
\end{figure}

\begin{table}[H]\centering
\caption{Approximated Magnus-2, Hermite spectral on $\R^3$,
$499^3$ DoFs, complex128, $T{=}1$, GH200. $t$\,(s) is the GPU time.}
\label{tab:magnus2_hermite}
{\small
\begin{tabular*}{\textwidth}{@{\extracolsep{\fill}}rrrrrrrrrrrrrr@{}}
\toprule
& \multicolumn{3}{c}{Yoshida ($M=1$)} & \multicolumn{3}{c}{Magnus-2 $M=3$} & \multicolumn{3}{c}{Magnus-2 $M=5$} & \multicolumn{3}{c}{Magnus-2 $M=7$} \\
\cmidrule(lr){2-4}\cmidrule(lr){5-7}\cmidrule(lr){8-10}\cmidrule(lr){11-13}
$\Delta t$ & Error & Rate & $t$ (s) & Error & Rate & $t$ (s) & Error & Rate & $t$ (s) & Error & Rate & $t$ (s) \\
\midrule
0.2   & 1.34e+0 & --    & 1.7  & 1.41e+0 & --   & 11   & 3.05e-1 & --   & 16   & 1.11e-1 & --   & 23 \\
0.125 & 1.21e+0 &  0.22 & 2.7  & 4.85e-1 & 2.27 & 17   & 3.59e-2 & 4.55 & 26   & 8.85e-3 & 5.38 & 36 \\
0.1   & 1.52e+0 & -1.02 & 3.4  & 1.11e-1 & 6.61 & 21   & 1.26e-2 & 4.69 & 33   & 3.51e-3 & 4.14 & 45 \\
0.05  & 5.48e-1 &  1.47 & 6.7  & 5.17e-3 & 4.42 & 42   & 7.60e-4 & 4.05 & 65   & 2.22e-4 & 3.98 & 91 \\
\bottomrule
\end{tabular*}
}
\end{table}

\bibliographystyle{siam}
\bibliography{refs}

\end{document}